\documentclass[11pt]{amsart}
\usepackage{amscd}
\usepackage{amsmath}
\usepackage{amsxtra}
\usepackage{amsfonts}
\usepackage{amssymb}

\newtheorem{theorem}{Theorem}[section]
\newtheorem{corollary}[theorem]{Corollary}

\newtheorem{lemma}[theorem]{Lemma}
\newtheorem{proposition}[theorem]{Proposition}

\theoremstyle{definition}
\newtheorem{definition}[theorem]{Definition}
\newtheorem{remark}[theorem]{Remark}

\newtheorem{example}[theorem]{Example}
\theoremstyle{remark}

\renewcommand{\theclaim}{\textup{\theclaim}}

\newtheorem*{acknowledgements}{Acknowledgements}

\numberwithin{equation}{section}

\def\openone

{\mathchoice

{\hbox{\upshape \small1\kern-3.3pt\normalsize1}}

{\hbox{\upshape \small1\kern-3.3pt\normalsize1}}

{\hbox{\upshape \tiny1\kern-2.3pt\SMALL1}}

{\hbox{\upshape \Tiny1\kern-2pt\tiny1}}}

\makeatletter

\newbox\ipbox

\newcommand{\ip}[2]{\left\langle #1\, , \,#2\right\rangle}
\newcommand{\diracb}[1]{\left\langle #1\mathrel{\mathchoice

{\setbox\ipbox=\hbox{$\displaystyle \left\langle\mathstrut
#1\right.$}

\vrule height\ht\ipbox width0.25pt depth\dp\ipbox}

{\setbox\ipbox=\hbox{$\textstyle \left\langle\mathstrut
#1\right.$}

\vrule height\ht\ipbox width0.25pt depth\dp\ipbox}

{\setbox\ipbox=\hbox{$\scriptstyle \left\langle\mathstrut
#1\right.$}

\vrule height\ht\ipbox width0.25pt depth\dp\ipbox}

{\setbox\ipbox=\hbox{$\scriptscriptstyle \left\langle\mathstrut
#1\right.$}

\vrule height\ht\ipbox width0.25pt depth\dp\ipbox}

}\right. }

\newcommand{\dirack}[1]{\left. \mathrel{\mathchoice

{\setbox\ipbox=\hbox{$\displaystyle \left.\mathstrut
#1\right\rangle$}

\vrule height\ht\ipbox width0.25pt depth\dp\ipbox}

{\setbox\ipbox=\hbox{$\textstyle \left.\mathstrut
#1\right\rangle$}

\vrule height\ht\ipbox width0.25pt depth\dp\ipbox}

{\setbox\ipbox=\hbox{$\scriptstyle \left.\mathstrut
#1\right\rangle$}

\vrule height\ht\ipbox width0.25pt depth\dp\ipbox}

{\setbox\ipbox=\hbox{$\scriptscriptstyle \left.\mathstrut
#1\right\rangle$}

\vrule height\ht\ipbox width0.25pt depth\dp\ipbox}

} #1\right\rangle}

\newcommand{\cj}[1]{\overline{#1}}

\newcommand{\bz}{\mathbb{Z}}

\newcommand{\br}{\mathbb{R}}

\newcommand{\bn}{\mathbb{N}}

\newcommand{\beq}{\begin{equation}}
\newcommand{\eeq}{\end{equation}}

\def\blfootnote{\xdef\@thefnmark{}\@footnotetext}

\renewcommand{\mod}{\operatorname{mod}}

\input xy
\xyoption{all}
\usepackage{amssymb}

\def\-{^{-1}}

\def\D{\mathcal{D}}

\def\ty{\emptyset}

\begin{document}

\title[Hadamard triples generate self-affine spectral measures]{Hadamard triples generate self-affine spectral measures}
\author{Dorin Ervin Dutkay}

\address{[Dorin Ervin Dutkay] University of Central Florida\\
	Department of Mathematics\\
	4000 Central Florida Blvd.\\
	P.O. Box 161364\\
	Orlando, FL 32816-1364\\
U.S.A.\\} \email{Dorin.Dutkay@ucf.edu}

\author{John Haussermann}

\address{[John Haussermann] University of Central Florida\\
	Department of Mathematics\\
	4000 Central Florida Blvd.\\
	P.O. Box 161364\\
	Orlando, FL 32816-1364\\
U.S.A.\\} \email{jhaussermann@knights.ucf.edu}

\author{Chun-Kit Lai}

\address{[Chun-Kit Lai] Department of Mathematics, San Francisco State University,
1600 Holloway Avenue, San Francisco, CA 94132.}

 \email{cklai@sfsu.edu}

\thanks{}

\begin{abstract}
Let $R$ be an expanding matrix with integer entries and let $B,L$ be finite integer digit sets so that $(R,B,L)$ form a Hadamard triple on ${\br}^d$ in the sense that the matrix
$$
\frac{1}{\sqrt{|\det R|}}\left[e^{2\pi i \langle R^{-1}b,\ell\rangle}\right]_{\ell\in L,b\in B}
 $$
is unitary. We prove that the associated fractal self-affine measure $\mu = \mu(R,B)$ obtained by an infinite convolution of  atomic measures
$$
\mu(R,B) = \delta_{R^{-1} B}\ast\delta_{R^{-2}B}\ast\delta_{R^{-3}B}\ast...
$$
is a spectral measure, i.e.,  it admits an orthonormal basis of exponential functions in $L^2(\mu)$. This settles a long-standing conjecture proposed by Jorgensen and Pedersen and studied by many other authors. Moreover, we also show that if we relax the Hadamard triple condition to an almost-Parseval-frame condition, then we obtain a sufficient condition for a self-affine measure to admit Fourier frames.
\end{abstract}
\maketitle \tableofcontents
\section{Introduction}

\subsection{Fuglede's Problems}
As it is well known, Fourier discovered that the exponential functions $\{e^{2\pi i \langle n,x\rangle}: n\in {\mathbb Z}^d\}$ form an orthonormal basis for $L^2([0,1]^d)$ and his discovery is now one of the fundamental pillars in modern mathematics. It is natural to ask what other measures have this property, that there is a family of exponential functions which form an orthonormal basis for their $L^2$-space?

 Let $\mu$ be a Borel probability measure on ${\mathbb R}^d$ and let $\langle\cdot,\cdot\rangle$ denote the standard inner product on $\br^d$. We say that $\mu$ is a {\it spectral measure} if there
exists a countable set $\Lambda\subset {\mathbb R}^d$ called the {\it spectrum} of $\mu$ such  that
$E(\Lambda): = \{e^{2\pi i \langle\lambda,x\rangle}:
\lambda\in\Lambda\}$ is an orthonormal basis for $L^2(\mu)$. Suppose that the Fourier transform of $\mu$ is defined to be
 $$
\widehat{\mu}(\xi)= \int e^{-2\pi i \langle\xi,x\rangle}d\mu(x).
 $$
 It is straightforward to verify that a measure is a spectral measure with spectrum $\Lambda$ if and only if the following two conditions are satisfied:
 \begin{enumerate}
\item (Orthogonality) $ \widehat{\mu}(\lambda-\lambda')=0$ for all distinct $\lambda,\lambda'\in\Lambda$ and
\item (Completeness) If for $f\in L^2(\mu)$, $\int f(x)e^{-2\pi i \langle\lambda,x\rangle}d\mu(x)=0$ for all $\lambda\in\Lambda$, then  $f=0$.
    \end{enumerate}
    Furthermore, we say that a Lebesgue measurable subset $\Omega$ of $\br^d$ is a {\it spectral set} if the corresponding Lebesgue measure supported on $\Omega$, $\chi_{\Omega} dx$, is a spectral measure.  In this paper, we are interested in the following question

\medskip

{\bf (Q1)} When is a Borel probability measure $\mu$ spectral?

\medskip

%

The question was first studied by  Fuglede \cite{Fug74} in 1974 while he was working on a problem by Segal on the existence of {\it commuting} extensions of the partial differential operators on domains of $\br^d$. Fuglede proved that the domains $\Omega$ under some regularity condition for which such extensions exist are exactly those with the property that there exists an orthogonal exponential  basis for $L^2(\Omega)$, with Lebesgue measure. The regularity condition was removed later by Pedersen \cite{P87}.  In the same paper, Fuglede proposed his famous conjecture:

\medskip

\noindent
{\bf Fuglede's Conjecture:}~~A measurable set $\Omega$ is a spectral set in ${\mathbb R}^d$ if and only if $\Omega$ tiles ${\mathbb R}^d$ by translation.

\medskip

Fuglede's Conjecture has been studied by many authors, e.g., Jorgensen, Pedersen, Lagarias, {\L}aba, Kolountzakis, Matolcsi, Iosevich, Tao, Wang and others (\cite{JP98,MR1700084,IKT1,IKT2,Ko1,MR2237932,La,LRW,LaWa96,LaWa97,Tao04}), but it had baffled experts for 30 years until Terence Tao \cite{Tao04} constructed the first counterexample, a spectral set which is not a tile in ${\mathbb R}^d$, $d\geq5$. The example and technique were refined later to disprove the conjecture in both directions on ${\mathbb R}^d$ for $d\geq3$ \cite{MR2159781,MR2237932,MR2264214}. The conjecture is still open in dimensions $d=1$ and $d=2$.

\medskip

Although Fuglede's Conjecture in its original form has been disproved, there is a clear connection between spectral sets and tilings, but the precise correspondence is still a mystery.  Furthermore, spectral sets are a particular case of a broader class of problems concerning the existence and construction of families of complex exponential functions that form either Riesz bases or, more generally, Fourier frames \cite{K15,KN15,NOU}. Also, it is known that Fuglede's conjecture is true under some additional assumptions and in some other groups \cite{IMP15}, and it is related to the construction of Gabor and wavelet bases \cite{LW,Y}. We will refer to the problems concerning spectral measures and their relation to translational tilings as the {\it Fuglede problem}.

\bigskip

\subsection{Fractal Spectral measures and Main Results}
Another major advance in the study of the Fuglede problem was the discovery that fractal singular measures can also be spectral. This opened up a new possibility of applying the well-developed Fourier analysis techniques to certain classes of fractals.

 \medskip

 In 1998, Jorgensen and Pedersen \cite{JP98} constructed the first example of a singular, non-atomic spectral measure. The measure is the Hausdorff measure supported on a Cantor set, where the scaling factor is 4 and the digits are 0 and 2; we call them the one-fourth Cantor measure/set. The spectrum for this measure is the set
$$\Lambda:=\left\{\sum_{k=0}^n 4^k l_k : l_k\in\{0,1\}, n\in\bn\right\}.$$

They also proved that the usual Middle Third Cantor measure is non-spectral. The Fourier series on the one-fourth Cantor measure were studied by Strichartz who proved in \cite{MR2279556} that they have much better convergence properties than their classical counterparts: the Fourier series associated to continuous functions converge uniformly and the Fourier series of $L^p$-functions converge in the $L^p$-norm.

\medskip

Following this discovery, many other examples of singular measures have been constructed, and the spectral property of various classes of fractal measures have been analyzed, see, e.g.,  \cite{JP98,LaWa02,Str98,Str00,DJ06,MR3163581,MR3273183,MR3302160,HL08,Dai, AHLau,GL} and the references therein. To the best of our knowledge, all these constructions have been based on the central idea of Hadamard matrices and {\it Hadamard triples}:
	
\begin{definition}\label{hada}
Let $R\in M_d({\mathbb Z})$ be an $d\times d$ expansive matrix (i.e., all eigenvalues have modulus strictly greater than 1) with integer entries. Let $B, L\subset{\mathbb Z}^d $ be  finite sets of integer vectors with $N:= \#B=\#L$ ($\#$ denotes the cardinality). We say that the system $(R,B,L)$ forms a {\it Hadamard triple} (or $(R^{-1}B, L)$ forms a {\it compatible pair}, as it is called in \cite{LaWa02} ) if the matrix
\begin{equation}\label{Hadamard triples}
H=\frac{1}{\sqrt{N}}\left[e^{2\pi i \langle R^{-1}b,\ell\rangle}\right]_{\ell\in L, b\in B}
\end{equation}
is unitary, i.e., $H^*H = I$.
\end{definition}

 The system $(R,B,L)$ forms a  Hadamard triple if and only if the Dirac measure $\delta_{R^{-1}B} = \frac{1}{\#B}\sum_{b\in B}\delta_{R^{-1}b}$ is a spectral measure on ${\mathbb R}^{d}$ with spectrum $L$. Infinite convolutions of rescaled discrete measures produce self-affine measures, which we define below.

\begin{definition}\label{defifs}
For a given expansive $d\times d$ integer matrix $R$ and a finite set of integer vectors $B$ with $\#B =: N$, we define the {\it affine iterated function system} (IFS)
 $$\tau_b(x) = R^{-1}(x+b),\quad ( x\in \br^d, b\in B).$$ The {\it self-affine measure} (with equal weights) is the unique probability measure $\mu = \mu(R,B)$ satisfying
\begin{equation}\label{self-affine}
\mu(E) = \sum_{b\in B} \frac1N \mu (\tau_b^{-1} (E)),\mbox{ for all Borel subsets $E$ of $\br^d$.}
\end{equation}
This measure is supported on the {\it attractor} $T(R,B)$ which is the unique compact set that satisfies
$$
T(R,B)= \bigcup_{b\in B} \tau_b(T(R,B)).
$$
The set $T(R,B)$ is also called the {\it self-affine set} associated with the IFS. It can also be described as
$$T(R,B)=\left\{\sum_{k=1}^\infty R^{-k}b_k : b_k\in B\right\}.$$
 One can refer to \cite{Hut81} and \cite{Fal97} for a detailed exposition of the theory of iterated function systems.
\end{definition}

For a given integral expanding matrix $R$ and  a simple digit set $B$ for $R$. We let
 \begin{equation}\label{B_n}
B_n := B+RB+R^{2}B+...+R^{n-1}B = \left\{\sum_{j=0}^{n-1}R^jb_j: b_{j}\in B\right\}.
\end{equation}
\begin{equation}\label{L_n}
L_n^T := L+R^TL+(R^T)^{2}L+...+(R^T)^{n-1}L = \left\{\sum_{j=0}^{n-1}(R^T)^j\ell_j: \ell_{j}\in L\right\}.
\end{equation}
Another important description of the self-affine measure $\mu(R,B)$ is as the infinite convolution of discrete measures
\begin{equation}\label{infinite convolution}
\begin{aligned}
\mu(R,B) =& \delta_{R^{-1}B}\ast\delta_{R^{-2}B}\ast\delta_{R^{-3}B}\ast...\\
=&\mu_n\ast\mu_{>n},
\end{aligned}
\end{equation}
where
$$
\mu_n =  \delta_{R^{-1}B}\ast\delta_{R^{-2}B}\ast...\ast\delta_{R^{-n}B} = \delta_{R^{-n}(B_n)}
$$
and $\mu_{>n} = \delta_{R^{-(n+1)}B}\ast\delta_{R^{-(n+2)}B}\ast... = \mu ((R^T)^{-n}(\cdot)) $ by self-similarity. For a finite set $A$ in $\br^d$,
$$\delta_A:=\frac{1}{\#A}\sum_{a\in A}\delta_a,$$
where $\delta_a$ is the Dirac measure at $a$.

\medskip

Suppose that $(R,B,L)$ is a Hadamard triple. Then $(R^{k},B, (R^T)^{k-1}L)$ are Hadamard triples for all $k$. Hence, each factor $\delta_{R^{-k}B}$ is a spectral measure. Moreover, because $R$ and $B$ have integer entries, we can see that all $\mu_n$ are spectral. Hence, it is natural to conjecture that the weak limit $\mu$ of  $\mu_n$ is spectral:

\medskip

\noindent
{\bf Conjecture:}~~ Suppose that $(R,B,L)$ forms a Hadamard triple. Then the self-affine measure $\mu(R,B)$ is a spectral measure.

\medskip

This conjecture has been proposed since Jorgensen and Pedersen's first discovery of spectral singular measures. It was first proved on ${\mathbb R}^1$ by Laba and Wang \cite{LaWa02} and later refined in \cite{DJ06}. The situation becomes more complicated when $d>1$. Dutkay and Jorgensen showed that the conjecture  is true if $(R,B,L)$ satisfies a technical condition called the {\it reducibility condition} \cite{DJ07d}. The conjecture is true under some additional assumptions, introduced by Strichartz \cite{Str98,Str00}. Some low-dimensional special cases were also considered by Li \cite{MR3163581,MR3302160}. In this paper, one of our main objectives is to prove that this conjecture is true, Hadamard triples always generate self-affine {\it spectral} measures.

\begin{theorem}\label{thmain}
Let $(R,B,L)$ be a Hadamard triple. Then the self-affine measure $\mu(R,B)$ is spectral.
\end{theorem}

\medskip

\subsection{Outline of the Proof} Throughout the paper, we will assume, without loss of generality, that $0\in B\cap L$. The proof of Theorem \ref{thmain} involves three main steps and each individual step is of independent interest.

\medskip

\noindent {\bf Step 1: The No-Overlap condition}

\medskip

\begin{definition}\label{nooverlap}
 We say that the self-affine measure $\mu = \mu(R,B)$ in Definition \ref{defifs} satisfies the {\it no-overlap condition or measure disjoint condition}  if
$$
\mu(\tau_{b}(T(R,B))\cap \tau_{b'}(T(R,B)))=0, \ \text{for all } b\neq b'\in B.
$$
We say that $B$ is {\it a simple digit set for $R$}  if distinct elements of $B$ are not congruent $(\mod R(\bz^d))$.
 \end{definition}
 It is easy to verify that $B$ must be a simple digit set for $R$ if $(R,B,L)$ is a Hadamard triple. We will prove that if the digit set $B$ is simple, then the no-overlap condition is satisfied.  The no-overlap condition is related to the open set condition (OSC) and strong open set condition (SOSC).

\medskip
\begin{definition}\label{OSC}
We say that the iterated function system $\{\tau_b\}_{b\in B}$ satisfies the {\it open set condition} (OSC) if there exists a non-empty open set $U$ such that
$$
\tau_b(U)\cap \tau_{b'}(U) = \emptyset, \ \mbox{and} \ \bigcup_{b\in B}\tau_b(U)\subset U.
$$
The iterated function system $\{\tau_b\}_{b\in B}$ satisfies the {\it strong open set condition} (SOSC) if we can furthermore choose the open set $U$ such that $U\cap T(R,B)\neq \emptyset$.
\end{definition}
These conditions have been well-studied in the case of self-similar measures for which $R = rO$ for some $r>1$ and orthogonal matrix $O$ (see e.g. \cite{Sc,LW93}), but we did not find any such results in the literature for the case self-affine measures. We will first prove the no-overlap condition for the self-affine measure in our interest.

   \begin{theorem}\label{th1.0}
	Let $R$ be a $d\times d$ expansive integer matrix and let $B$ be a simple digit set for $R$. Then the affine iterated functions system associated to $R$ and $B$ satisfies the OSC, SOSC and the no-overlap condition.
\end{theorem}

In the study of spectral measures, the no overlap condition for a self-affine measures is particularly important since it guarantees that $\mu (\tau_bT(R,B))= 1/N$ and its $k$-th level iterates will have measure $1/N^k$. With the help of this theorem, we can also compute $\int |f|^2d\mu$ for the set of step functions $f$ on the self-affine set $T(R,B)$.

\bigskip

After establishing the no-overlap condition, we can start the proof of Theorem \ref{thmain}. The mutual orthogonality of the exponential functions is not difficult to show. The main challenge is to establish the completeness of the set of exponential functions. We consider the following periodic zero set of the Fourier transform:
 \begin{equation}
{\mathcal Z} := \{\xi\in{\mathbb R}^d: \widehat{\mu}(\xi+k)=0, \ \mbox{for all} \  k\in{\mathbb Z}^d\}
 \label{eqz}
 \end{equation}
We will divide our proof into two cases: (i) ${\mathcal Z}= \ty$ and (ii) ${\mathcal Z} \ne\ty$.

\medskip

\noindent {\bf Step 2: ${\mathcal Z} = \ty$}

\medskip

This case is easier to handle. For a Hadamard triple $(R,B,L)$ and a sequence of positive integers $n_k$, we let $m_k = n_1+...+n_k$. The self-affine measure can be rewritten as
$$
\mu(R,B) = \delta_{R^{-m_1}B_{n_1}}\ast\delta_{R^{-m_2}B_{n_2}}\ast...\ast\delta_{R^{-m_k}B_{n_k}}\ast...
$$
Then we note that if we have another set $J_{n_k}$ of integer vectors, with $J_{n_k}\equiv L^T_{n_k}$ (mod $(R^T)^{n_k}({\mathbb Z}^d)$), then $(R^{n_k}, B_{n_k},J_{n_k})$ still form Hadamard triples.  Using this, we can produce many mutually orthogonal sets of exponential functions with frequencies given by:
\begin{equation}\label{eqLambda_k}
\Lambda_k = J_{n_1}+ (R^T)^{m_1}J_{n_2}+ (R^T)^{m_2} J_{n_3}+...+ (R^T)^{m_{k-1}} J_{n_{k}},
\end{equation}
\begin{equation}\label{eqlambda}
\ \Lambda = \bigcup_{k=1}^{\infty}\Lambda_k.
\end{equation}
We will show that under the assumption ${\mathcal Z}=\ty$, we can pick such a set $\Lambda$ that is indeed also complete, so it is a spectrum. In fact, we have
\begin{theorem}\label{th1.1}
Suppose that $(R,B,L)$ forms a Hadamard triple and $\mu = \mu(R,B)$ is the associated self-affine measure. Then the following are equivalent
	\begin{enumerate}
   \item ${\mathcal Z}=\emptyset$,
		\item  $\mu$ has a spectrum in $\bz^d$.
 \end{enumerate}
In particular, if ${\mathcal Z}=\emptyset$, then  $\mu$ is a spectral measure.
\end{theorem}

 Our method for the proof of the completeness of the set of exponential functions differs from all the other existing proofs in literature, see, e.g., \cite{LaWa02,Str98,Str00,MR3055992,Dai}, where the completeness is established by checking the Jorgensen-Pedersen criterion (i.e. $\sum_{\lambda\in\Lambda}|\widehat{\mu}(\xi+\lambda)|^2=1$). Our proof of this theorem with ${\mathcal Z} = \ty$ relies on an approach from matrix analysis which exploits the isometry property of Hadamard matrices, i.e. $\|H{\bf w}\| = \| {\bf w}\|$. This allows us to show that the frame inequalities are satisfied for the set of all step functions, and then, by a density argument, the collection of the exponential functions has to be complete. This argument also gives us sufficient conditions to consider another famous question as to whether Fourier frames can exist for non-spectral self-affine measures, such as the Middle Third Cantor measure (See Subsection \ref{Fourier frame}).

\medskip

\noindent {\bf Step 3: ${\mathcal Z} \neq \ty$}

\medskip

To complete the proof of Theorem \ref{thmain} we have to consider the case ${\mathcal Z}\neq \emptyset$. When $\mathcal Z\neq \ty$, there is an exponential function $e^{2\pi i \langle\xi, x\rangle}$ that is orthogonal to every exponential function with integer frequencies. This implies that none of the subsets of integers can be complete and hence none of the sets $\Lambda$ in \eqref{eqlambda} can be complete.

\medskip

It is possible that ${\mathcal Z}\ne \emptyset$. The simplest example is to consider the interval $[0,2]$ which is generated by the IFS $\tau_0(x) = \frac12x$ and $\tau_2(x)= \frac{1}{2}(x+2)$. In this case, ${\mathcal Z} = {\mathbb Z}+\frac12$. However, this is rather trivial since the greatest common divisor (gcd) of $B = \{0,2\}$ is not 1. In fact, by some conjugation, we can assume the smallest $R$-invariant lattice containing all sets $B_n$, denoted by ${\mathbb Z}[R,B]$ is ${\mathbb Z}^d$. On ${\mathbb R}^1$, it is equivalent to gcd$(B)=1$ and we can settle this case using the result for ${\mathcal Z}= \emptyset$. However, this simple situation ceases to exist when $d>1$ and we can find spectral self-affine measures with ${\mathcal Z}\neq \emptyset$ and ${\mathbb Z}[R,B] = {\mathbb Z}^d$.

\medskip

 To settle this case, our strategy is to identify ${\mathcal Z}$ as an invariant set of some dynamical system, and use the techniques in \cite{CCR}. By doing so, we are able to show that in the case when ${\mathcal Z}\neq \emptyset$ the digit set $B$ will be reduced to a {\it quasi product-form}. Our methods are also similar to the ones used in \cite{LW2}. However, as $B$ is not a complete set of representatives (mod $R({\mathbb Z}^d)$) (as it was in \cite{LW2}), several additional adjustments will be needed.  From the quasi-product form structure obtained, we construct the spectrum directly by induction on the dimension $d$ with the help of the Jorgensen-Pedersen criterion.

 \medskip

\subsection{Fourier frames}\label{Fourier frame}
We say that the self-affine measure $\mu=\mu(R,B)$  admits a Fourier frame $E(\Lambda) = \{e^{2\pi i \langle\lambda,x\rangle}:\lambda\in\Lambda\}$ if there exists $0<A\leq B<\infty$ such that
$$
A\|f\|^2\leq \sum_{\lambda\in\Lambda}|\int f(x)e^{-2\pi i \langle\lambda,x\rangle}d\mu(x)|^2\leq B\|f\|^2, \ \forall \ f\in L^2(\mu).
$$
It is clear that the concept of Fourier frames is a natural generalization of exponential orthonormal bases. Whenever Fourier frames exist, $\mu$ is called a {\it frame spectral measure} and $\Lambda$ is called a {\it frame spectrum}. Frames on a general Hilbert space were introduced by Duffin and Schaeffer \cite{DS52} and are now a fundamental research area in applied harmonic analysis, which is developing rapidly both in theory and in applications. In theory, Fourier frames are related to de Brange's theory in complex analysis \cite{OSANN,Seip2}. In application, people regard frames as ``overcomplete bases" and because of their redundancy, the reconstruction process is more robust to errors in data and it is now widely used in signal transmission and reconstruction.  Reader may refer to \cite{Chr03} for the background of the general frame theory.

\medskip

For the measures which are non-spectral, it is natural to ask the following question.

\medskip

{\bf (Q2)} Can a non-spectral fractal measure still admit some Fourier frames?

\medskip

Some of the fundamental properties of Fourier frames were investigated in \cite{HLL11,DHSW11,DHW11b,DL14}. This question was first proposed by Strichartz \cite[p.212]{Str00}. In particular, there have been discussions whether, specifically, the one-third Cantor measure is frame spectral. We introduce the following condition as a natural generalization of Hadamard triples.

 \begin{definition}\label{ALmost_DEF}
 We say that the pair $(R,B)$ satisfies the {\it almost-Parseval-frame condition} if for any $\epsilon>0$, there exists $n$ and a subset $J_n\subset{\mathbb Z}^d$ such that
 \begin{equation}\label{ALmost}
(1-\epsilon)\sum_{b\in B_n}|w_b|^2\leq \sum_{\lambda\in J_n}\left|\frac{1}{\sqrt{N^n}}\sum_{b\in B_n}w_be^{-2\pi i \langle R^{-n}b, \lambda\rangle}\right|^2\leq (1+\epsilon)\sum_{b\in B_n}|w_b|^2
\end{equation}
for all ${\bf w} = (w_b)_{b\in{B_n}}\in{\mathbb C}^{N^n}$. Equivalently,
$$
(1-\epsilon)\|{\bf w}\|^2\leq \|F_n {\bf w}\|^2\leq (1+\epsilon)\|{\bf w}\|^2
$$
where $F_n= \left[\frac{1}{\sqrt{N^n}}e^{-2\pi i \langle R^{-n}b, \lambda\rangle}\right]_{\lambda\in J_n,b\in B_n}$ and $\|\cdot\|$ denotes the Euclidean norm.
 \end{definition}

Hadamard triples do satisfy this condition (even with $\epsilon=0$) and we  prove:

\begin{theorem}\label{th1.2}
Suppose that $B$ is simple digit set for $R$ and that $(R,B)$ satisfies the almost-Parseval-frame condition. Suppose that  $\mathcal Z=\emptyset$. Then the self-affine measure $\mu=\mu(R,B)$  admits a Fourier frame $E(\Lambda) = \{e^{2\pi i \langle\lambda,x\rangle}:\lambda\in\Lambda\}$ with $\Lambda\subset{\mathbb Z}^d$.
\end{theorem}

\medskip

In fact, we will see that a natural geometric condition will guarantee that $\mathcal Z=\emptyset$. This condition is satisfied for  the one-third Cantor measure. The construction of Fourier frames now turns into a problem of matrix analysis, which is to construct finite sets $J_n$ so that the almost-Parseval-frame condition holds. At this time, we were unable to give a full solution. However, the recent solution of the Kadison-Singer conjecture \cite{MSS} enabled Nitzan, Olevskii, Unlanovskii \cite{NOU} to construct Fourier frames on unbounded sets of finite measures. One of their lemmas gives us a weak solution:

\begin{proposition}\label{prop1.3}
Suppose that $B$ is simple digit set for $R$. There exist  universal constants $0< c_0< C_0<\infty$ such that for all $n$, there exists $J_n$ such that
$$
c_0\sum_{b\in B_n}|w_b|^2\leq \sum_{\lambda\in J_n}\left|\frac{1}{\sqrt{N^n}}\sum_{b\in B_n}w_be^{-2\pi i \langle R^{-n}b, \lambda\rangle}\right|^2\leq C_0\sum_{b\in B_n}|w_b|^2
$$
for all $(w_b)_{b\in B_n}\in{\mathbb C}^{N^n}$.
\end{proposition}

\medskip

In the proof in Theorem \ref{th1.2}, the idea is to concatenate the sets $J_n$ in order to obtain the frame spectrum $\Lambda$, but for this we need $\epsilon$ in Definition \ref{ALmost_DEF} to be arbitrarily small . We cannot use the same method using just Proposition \ref{prop1.3}. It would be nice if we could construct an {\it increasing} sequence of sets $J_n$, because then the frame spectrum can be obtained as the union of the sets $J_n$. In any case, this proposition sheds some light on the plausibility of the almost-Parseval-frame condition. In fact, if we consider fractal measures that are not self-affine, and we allow some flexibility in the choice if the contraction ratios at different levels, then it can be proved that non-spectral fractal measure with Fourier frames do exist via the almost-Parseval-frame condition \cite{LW16}.

\medskip

We organize our paper as follows: In Section 2, we prove the no-overlap condition for self-affine measures. In Section 3, we study the almost-Parseval-frame condition and concatenation of Hadamard triples. In Section 4, we will prove Theorem \ref{thmain} under the assumption ${\mathcal Z}=\emptyset$. In Section 5, we further reduce our problem to ${\mathbb Z}[R,B] = {\mathbb Z}^d$ and we prove Theorem \ref{thmain}  on ${\mathbb R}^1$. In Section 6, we introduce the techniques from \cite{CCR}. In Section 7, we use these techniques to show that $B$ must be of quasi-product form if ${\mathcal Z}\ne \ty$. In Section 8, we prove Theorem \ref{thmain} in full generality. In Section 9, We study Fourier frames on self-affine measures using the almost-Parseval-frame condition. We end the paper with some open problems in Section 10. Finally, an appendix is given to sharpen the frame bound in Section 9.

%
%

\medskip

		
%
%
%
%
%
%
%
%
%
%
%

\section{The no-overlap condition}

\medskip

This section is devoted to study the no-overlap condition from fractal geometry point of view and note that no Hadamard triple assumption is imposed. Throughout the section, we  will fix the affine IFS given by  an expansive matrix $R$ with integer entries and a simple digit set $B$ for $R$ and $0\in B$. The map is defined by
$$
\tau_b(x) = R^{-1}(x+b), \quad(b\in B,x\in\br^d).
$$
and $T = T(R,B)$ is its attractor.

We introduce some multi-index notation to describe our IFS. Let $B^n= B\times B...\times B$ ($n$ copies) and $\Sigma = \bigcup_{n=1}^{\infty}B^n$. For each ${\bf b} = (b_1,...,b_n)\in B^n$,
$$
\tau_{{\bf b}} (x) =\tau_{b_1}\circ...\circ\tau_{b_n}(x).
$$
Also for any set $A\subset{\mathbb R}^d$, we define $A_{{\bf b}} =\tau_{{\bf b}}(A)$. Given a set of probabilities $0<p_b<1$, $b\in B$, ($\sum_{b\in B}p_b=1$), the associated {\it self-affine measure} is the unique Borel probability measure supported on $T(R,B)$ satisfying the invariance identity
\begin{equation}\label{self-affine measure}
\mu = \sum_{b\in B} p_b \mu_{b},
\end{equation}
where we define $\mu_b (E) = \mu (\tau_{b}^{-1}(E))$, for all Borel sets $E$, see \cite{Hut81}. By iterating this identity, we have
$$
\mu = \sum_{{\bf b}\in B^n} p_{\bf b} \mu_{{\bf b}},
$$
where $p_{{\bf b}} =p_{b_1}....p_{b_n}$ and $\mu_{{\bf b}} (E) = \mu( \tau_{\bf b}^{-1}(E))$ for all Borel sets $E$ and ${\bf b}= (b_1,...,b_n)$. With the definition of OSC and SOSC in Definition \ref{OSC}, the following theorem was proved by He and Lau \cite[Theorem 4.4]{HeL08}, see also \cite{Sc} for self-similar IFSs.

 \begin{theorem}\cite{HeL08}\label{th1.2hl}
 For a self-affine IFS,  the OSC and SOSC are equivalent.
 \end{theorem}

\medskip

For any set $F$, we denote by  $\overline{F}$, $F^{o}$, $\partial F$ the closure, interior and its boundary respectively.  The following theorem shows that the strong open set condition implies the no-overlap condition. Its proof is motivated by \cite[Lemma 2.2]{DeHL}.

\begin{theorem}\label{proposition2.2}
Suppose that the IFS satisfies the strong open set condition with the open set $U$. Then for the self-affine measure in \eqref{self-affine measure}, one has $\mu(U)=1$ and $\mu(\partial U)=0$. Moreover, $\mu$ satisfies the no-overlap condition.
\end{theorem}

\medskip

\begin{proof}
As $T(R,B)\cap U\neq \emptyset$, we can find $x_0\in T(R,B)\cap U$ and $\delta>0$ such that $B_{\delta}(x_0)\subset U$. In particular, there exists ${\bf b_0}\in B^n$, for some $n$ such that $\tau_{\bf b_0}(T(R,B))\subset B_{\delta}(x_0)\subset U$. Let $ C = B^{n}\setminus\{{\bf b_0}\}$  and let
$$
E_k = \bigcup_{{\bf b}\in B^{nk}\setminus C^k} \tau_{{\bf b}} (T(R,B)).
$$
For any ${\bf b} = ({\bf b}_1,...,{\bf b_k})\in B^{nk}\setminus C^k$, there exists at least one $1\leq s\leq k$ such that ${\bf b}_s = {\bf b}_0$. Then
$$
\tau_{{\bf b}}(T(R,B))\subset \tau_{{\bf b}_1...{\bf b}_{s-1}} (\tau_{{\bf b}_0}(T(R,B)))\subset\tau_{{\bf b}_1...{\bf b}_{s-1}} (U).
$$
As $U$ satisfies the open set condition for the IFS $\{\tau_b: b\in B\}$, we have  $\tau_{{\bf b}_1...{\bf b}_{s-1}} (U)\subset U.$ Hence, $E_k\subset U$. Now,
$$
\begin{aligned}
1\geq \mu (U)\geq \mu(E_k) =& \sum_{{\bf b}\in B^{nk}}p_{{\bf b}} \mu(\tau_{{\bf b}}^{-1}(E_k))\\
\geq &\sum_{{\bf b}\in B^{nk}\setminus C^k}p_{{\bf b}} \mu(\tau_{{\bf b}}^{-1}(E_k))\\
\geq& \sum_{{\bf b}\in B^{nk}\setminus C^k}p_{{\bf b}} \mu(\tau_{{\bf b}}^{-1}(\tau_{{\bf b}}(T(R,B)))\\
=& \sum_{{\bf b}\in B^{nk}\setminus C^k}p_{{\bf b}}\\
=& \sum_{{\bf b}\in B^{nk}}p_{{\bf b}}-\sum_{{\bf b}\in C^{k}}p_{{\bf b}}\\
=& 1- (\sum_{{\bf b}\in C}p_{{\bf b}})^k = 1-(1-p_{\bf b_0})^k.
\end{aligned}
$$
As $1-p_{\bf b_0}>0$,  $(1-p_{\bf b_0})^k$ tends to 0 as $k$ tends to infinity. This shows that $\mu(U) =1$. As $T(R,B)\subset \overline{U}$ (because $\cup_b\tau_b(\overline U)\subset \overline U$), we must have $\mu (\overline{U})=1$ and $\mu(\partial U)=0$.

\medskip

For the no-overlap condition, we note that $T(R,B)\subset \overline{U}$. Then $T(R,B)_b\subset ({\overline U})_b = \overline{U_b}$. Hence,
$$
\tau_b(T(R,B))\cap \tau_{b'}(T(R,B))\subseteq  \overline{U_b}\cap \overline{U_{b'}} = (U_b\cap \partial (U_b)) \cup (U_{b'}\cap \partial (U_{b'})).
$$
But $U$ is an open set satisfying the OSC, so $\tau_b(T(R,B))\cap \tau_{b'}(T(R,B))\subseteq \partial (U_b)\cap \partial (U_{b'})$.  The no overlap condition will follow if  we can show that $\mu(\partial{U_b})=0$ for all $b\in B$.

\medskip

Suppose on the contrary that $\mu(\partial{U_b})>0$, we apply \eqref{self-affine measure}) and obtain
$$
0<\mu (\partial{U_b}) = \sum_{b'\in B}p_b \mu (\tau_{b'}^{-1}(\partial{U_b}))
$$
This implies that for some $b'$, $\mu (\tau_{b'}^{-1}(\partial{U_b}))>0$. But $\tau_{b'}^{-1}(\partial{U_b}) = \partial U + Rb-b'$ and $\mu$ is supported  essentially on  $U$, so we have
$$
\mu ((\partial U + Rb-b')\cap U)>0.
$$
As $U$ is open, $U\cap (U+Rb-b')\neq \emptyset$. This implies that $\tau_b\tau_0(U)\cap \tau_0\tau_{b'}(U)\neq \emptyset$ and this contradicts to the open set condition for $U$  (by a translation we can always assume $0\in B$). Hence, $\mu(\partial{U_b})=0$ and this completes the proof.
\end{proof}

\medskip


\begin{proposition}\label{proposition2.3}
Let $R$ be a $d\times d$ integer expansive matrix and $B$ be a simple digit set for $R$. Suppose that $\overline{B}\supset B$ is a complete representative class   ($\mod R{\mathbb Z}^d$). Then  the open set condition for the IFS $\{\tau_b\}_{b\in B}$ is satisfied with open set $T(R, \overline{B})^\circ$.
%
\end{proposition}

\medskip

\begin{proof}
The statement that the open set condition is satisfied for the IFS $\{\tau_b\}_{b\in B}$ with open set $T^\circ(R, \overline{B})$ is probably known, but we present the proof for completeness. Let $T = T(R,\overline{B})$ and note that $T = \bigcup_{b\in \overline{B}}\tau_b(T)$. By taking the interior, we have $T^\circ\supset \bigcup_{b\in \overline{B}}\tau_b(T^\circ)\supset\bigcup_{b\in B}\tau_b(T^\circ)$. Also $T^\circ$ is non-empty, by \cite{LW1}. To see that $\tau_{b}(T^\circ)\cap \tau_{b'}(T^\circ) =\emptyset$, we take Lebesgue measure on  the invariance identity and obtain
$$
\mbox{Leb}(T) = \mbox{Leb}\left(\bigcup_{b\in \overline{B}}\tau_b(T)\right)\leq \sum_{b\in\overline{B}}\mbox{Leb}(\tau_b(T)) = \frac{\#\overline{B}}{|\det(R)|}\mbox{Leb}(T) = \mbox{Leb}(T).
$$
Here Leb$(T)$ denotes the Lebesgue measure of $T$ and $\#\overline{B} = |\det(R)|$ because $\overline{B}$ is a set of complete representatives (mod $R{\mathbb Z}^d$). $\mbox{Leb}(T)$ is non-zero, by \cite{LW1}. Hence,  $$\mbox{Leb}\left(\bigcup_{b\in \overline{B}}\tau_b(T)\right)= \sum_{b\in\overline{B}}\mbox{Leb}(\tau_b(T))$$ and Leb($\tau_b(T)\cap \tau_{b'}(T)$)=0. This implies that $\tau_b(T^\circ)\cap \tau_{b'}(T^\circ)= \emptyset$ since $\tau_b(T^\circ)\cap \tau_{b'}(T^\circ)$  is an open set.
\end{proof}
%

%
%
%

\bigskip

Theorem \ref{th1.0} follows readily from the results above.

\begin{proof}[ Proof of Theorem \ref{th1.0}] Proposition \ref{proposition2.3} shows that the OSC is satisfied, Theorem \ref{th1.2hl} shows that the SOSC is satisfied and then Theorem \ref{proposition2.2} shows that the no-overlap condition holds.
\end{proof}
%
%
%

 \medskip

In the end of this section, we mention that unequally-weighted self affine measures  do not admit any Fourier frames, using one of our previous results.

\begin{theorem}
Let $R$ be an expansive matrix with integer entries and let $B$ be a simple digit set for $R$. Suppose that $\mu$ defined in \eqref{self-affine measure} admits a Fourier frames. Then all $p_b$ are equal.
\end{theorem}

\begin{proof}
Note that all these measures satisfies the no-overlap condition by Theorem \ref{th1.0}. By  \cite[Theorem 1.5]{DL14}, all $p_b$ are equal.
\end{proof}

Because of the previous theorem, for the remainder of the paper, we will assume that the self-affine measures have equal weights $p_b=\frac{1}{N}$.

\medskip
 \section{The almost-Parseval-frame conditions and Hadamard triples.}

%
%

\medskip

In this section, we study the almost-Parseval-frame condition in Definition \ref{ALmost_DEF}. First of all, we note that there is no loss of generality to assume $0\in J_n$, because we can replace $w_b$ by $w_be^{2\pi i \langle R^{-n}b, \lambda_0\rangle}$, and \eqref{ALmost} is satisfied with $J_n$ replaced by $J_n-\lambda_0$. 
 %
%
%
%
%

\medskip

\begin{proposition}\label{proposition2.1}
 Suppose that the pair $(R,B)$ satisfies the almost Parseval frame condition and $J_n\subset{\mathbb Z}^d$ is the set satisfying \eqref{ALmost}, with $\epsilon<1$. We have the following:

\medskip

(i)  The elements in $J_n$ have distinct residues modulo $(R^T)^n({\mathbb Z}^d)$.

\medskip

(ii) Let $\widehat{J_n}\equiv J_n$ ($\mod (R^T)^n{\mathbb Z}^d$), then $\widehat{J_n}$ also satisfies \eqref{ALmost}.

\end{proposition}


\begin{proof}
(i) Suppose on the contrary that we can find  $\lambda',\lambda''\in J_n$ such that $\lambda'$ and $\lambda''$ are in the same equivalence class modulo $(R^T)^n{\mathbb Z}^d$. Let $w_b= e^{2\pi i \langle R^{-n}b,\lambda''\rangle}$, for all $b\in B_n$,  and plug it in \eqref{ALmost}. From the upper bound, we have
$$
2N^n+\sum_{\lambda\in J_n\setminus\{\lambda',\lambda''\}}\left|\frac{1}{\sqrt{N^n}}\sum_{b\in B_n}w_be^{-2\pi i \langle R^{-n}b, \lambda\rangle}\right|^2\leq (1+\epsilon)N^n.
$$
This implies that
$$
\sum_{\lambda\in J_n\setminus\{\lambda',\lambda''\}}\left|\frac{1}{\sqrt{N^n}}\sum_{b\in B_n}w_be^{-2\pi i \langle R^{-n}b, \lambda\rangle}\right|^2\leq (\epsilon-1)N^n<0
$$
which is a contradiction.
(ii) follows immediately  from $\langle R^{-n}b,\lambda+(R^{T})^nk\rangle = \langle R^{-1}b,\lambda\rangle$ for all $b\in B_n$, $\lambda\in J_n$ and $k\in {\mathbb Z}^d$.
\end{proof}

Assuming that the almost-Parseval-frame condition is satisfied, we consider sequences $\epsilon_k<1$ such that $\sum_k\epsilon_k<\infty$ and let $n_k$ and $J_{n_k}$ be the associated quantities satisfying
$$
(1-\epsilon_k)\sum_{b\in B_{n_k}}|w_b|^2\leq \sum_{\lambda\in J_{n_k}}\left|\sum_{b\in B_{n_k}}\frac{1}{\sqrt{N^{n_k}}}w_be^{-2\pi i \langle R^{-n_k}b, \lambda\rangle} \right|^2\leq (1+\epsilon_k)\sum_{d\in\D_{n_k}}|w_d|^2.
$$
Letting $m_k = n_1+n_2+...+n_{k}$, we consider the $\Lambda_k$ and $\Lambda$ defined in \eqref{eqLambda_k} and \eqref{eqlambda}, i.e.,
$$
\Lambda_k = J_{n_1}+ (R^T)^{m_1}J_{n_2}+ (R^T)^{m_2} J_{n_3}+...+ (R^T)^{m_{k-1}} J_{n_{k}},
\  \Lambda = \bigcup_{k=1}^{\infty}\Lambda_k.
$$
 Note that the digit sets  $B_{m_1}\subset B_{m_2}\subset...$  satisfy
$$
B_{m_{k+1}} =  R^{n_{k+1}}B_{m_k}+ B_{n_{k+1}}, B_{m_1} = B_{n_1}.
$$

\begin{proposition}\label{Prop_Approach_2.1}
With the notations above, we have
$$
c_k\|{\bf w}\|^2\leq \sum_{\lambda\in{\Lambda_{{k}}}}\left|\frac{1}{\sqrt{N^{m_k}}}\sum_{b\in B_{m_k}}w_be^{-2\pi i \langle R^{-m_k}b,\lambda\rangle}\right|^2\leq C_k\|{\bf w}\|^2
$$
where $c_k = \prod_{j=1}^{k}(1-\epsilon_j)$ and $C_k = \prod_{j=1}^{k}(1+\epsilon_j)$.
\end{proposition}

\begin{proof}
We prove this by induction on $k$. The inequality for $k=1$ is the almost-Parseval-frame condition with $B_{n_1}$ and $J_{n_1}$. Assuming the inequality is proved for $k$, we now establish it for $k+1$. We consider the upper bound inequality. If $b\in B_{m_{k+1}}$ and $\lambda\in\Lambda_{k+1}$, we can write uniquely $b=R^{n_{k+1}}b_1+b_2$ and $\lambda = \lambda_1+(R^{T})^{m_{k}}\lambda_2$ where $\lambda_1\in{\Lambda}_k$, $\lambda_2\in J_{n_{k+1}}$, $b_1\in B_{m_k}$ and $b_2\in B_{n_{k+1}}$. For any vectors ${\bf w} = (w_{b})_{b\in B_{m_{k+1}}} = (w_{b_1b_2})_{b_1\in B_{m_{k}},b_2\in B_{n_{k+1}}}$, we have
$$
\begin{aligned}
&\sum_{\lambda\in\Lambda_{{k+1}}}\left|\frac{1}{\sqrt{N^{m_{k+1}}}}\sum_{b\in{B_{m_{k+1}}}}w_be^{-2\pi i \langle R^{-m_{k+1}}b,\lambda \rangle}\right|^2\\
=&\sum_{\lambda_1\in\Lambda_k}\sum_{\lambda_2\in J_{n_{k+1}}}\left|\frac{1}{\sqrt{N^{m_{k+1}}}}\sum_{b_2\in{B}_{n_{k+1}}}\sum_{b_1\in{B_{m_{k}}}}w_{b_1b_2}e^{-2\pi i   \langle R^{-m_{k+1}}(R^{n_{k+1}}b_1+b_2),\lambda_1+(R^{T})^{m_{k}}\lambda_2\rangle}\right|^2\\
=&\sum_{\lambda_1\in{\Lambda}_k}\sum_{\lambda_2\in J_{n_{k+1}}}\left|\frac{1}{\sqrt{N^{n_{k+1}}}}\sum_{b_2\in{B}_{n_{k+1}}}e^{-2\pi i \langle R^{-n_{k+1}}b_2,\lambda_2\rangle}\sum_{b_1\in{B_{m_{k}}}}\frac{1}{\sqrt{N^{m_{k}}}}w_{b_1b_2}e^{-2\pi i \langle R^{-m_k}b_1+R^{-m_{k+1}}b_2,\lambda_1\rangle}\right|^2\\
\leq& (1+\epsilon_{k+1})\sum_{\lambda_1\in{\Lambda}_k} \sum_{b_2\in{B}_{n_{k+1}}}\left|\frac{1}{\sqrt{N^{m_{k}}}}\sum_{b_1\in{B_{m_{k}}}}w_{b_1b_2}e^{-2\pi i \langle R^{-m_k}b_1+R^{-m_{k+1}}b_2,\lambda_1\rangle}\right|^2\\
=&(1+\epsilon_{k+1})\sum_{b_2\in{B}_{n_{k+1}}}\sum_{\lambda_1\in{\Lambda}_k}  \left|\frac{1}{\sqrt{N^{m_{k}}}}\sum_{b_1\in{B_{m_{k}}}}w_{b_1b_2}e^{-2\pi i \langle R^{-m_k}b_1,\lambda_1\rangle}\right|^2 \ (\mbox{as} \ \left|e^{-2\pi i \langle R^{-m_{k+1}}b_2,\lambda_1\rangle}\right|=1 )\\
\leq& (1+\epsilon_{k+1})C_{k}\sum_{b_2\in{B_{n_{k+1}}}}\sum_{b_1\in{B_{m_{k}}}}|w_{b_1b_2}|^2= C_{k+1}\|{\bf w}\|^2.\\
\end{aligned}
$$
The proof for the lower bound is similar.
\end{proof}

\medskip

Now, we turn to study Hadamard triples $(R,B,L)$ as defined in \eqref{Hadamard triples} in the introduction. We first remark that the elements of $B$ must be in distinct residue classed modulo $R({\mathbb Z}^d)$, because $H$ must have mutually orthogonal rows. This implies that
 \begin{equation}
 \sum_{\ell\in L} e^{2\pi i \langle R^{-1}(b-b'),\ell\rangle}=0, \ \mbox{if} \ b\neq b'.
 \end{equation}
 If $b= b'+ Rk$ for some $k\in {\mathbb Z}^d$, the sum above is equal to $\#L\neq 0$. Similarly, the elements $L$ must be in distinct residue class modulo $R^T{\mathbb Z}^d$.  As $H$ is a unitary matrix, it is clear that we have $\|H{\bf w}\| = \|{\bf w}\|$ for all ${\bf w}\in {\mathbb C}^N$. i.e.
$$
\sum_{\ell\in L} \left|\sum_{b\in B} w_{b} \frac{1}{\sqrt{N}}e^{-2\pi i \langle R^{-1}b,\ell\rangle}\right|^2 = \sum_{b\in B}|w_b|^2.
$$
From this, we will conclude in Corollary \ref{corollary2.3} that $(R,B)$ satisfies the almost-Parseval-frame condition (with $\epsilon=0$!). We also need to consider towers of Hadamard triples. Using the definition of $B_n$ in \eqref{B_n} and $L^T_n$ in \eqref{L_n}, from Proposition \ref{Prop_Approach_2.1}, we have the following corollary.

\begin{corollary}\label{corollary2.3}
(i) Suppose that $(R,B,L)$ is a Hadamard triple. Then for all $k\ge 1$, $(R^k, B_k, L^T_k)$ are Hadamard triples.

\medskip

(ii) Suppose that $(R^{n_i}, B_{n_i}, J_{n_i})$, $i=1,2,..$, are Hadamard triples, then for all $k\ge 1$, $(R^{m_k}, B_{m_k}, \Lambda_k)$ are Hadamard triples where $\Lambda_k$ are defined in \eqref{eqLambda_k}.
\end{corollary}

\begin{proof}
Suppose that $(R,B,L)$ is a Hadamard triple. Then we take $n_i=1$ and $J_{n_i} = L$. We have $\Lambda_k = L^T_k$. Proposition \ref{Prop_Approach_2.1} implies that
\begin{equation}\label{eq3.1a}
\sum_{\lambda\in{L^T_k}}\left|\sum_{b\in B_{k}}\frac{1}{\sqrt{N^k}}w_be^{-2\pi i \langle R^{-k}b,\lambda\rangle}\right|^2= \|{\bf w}\|^2, \ \forall {\bf w}\in {\mathbb C}^{N^k}.
\end{equation}
Similarly, if $(R^{n_i}, B_{n_i}, J_{n_i})$, $i=1,2,..$, are Hadamard triples, we also have
\begin{equation}\label{eq3.2a}
\sum_{\lambda\in{\Lambda_{{k}}}}\left|\sum_{b\in B_{m_k}}\frac{1}{\sqrt{N^{m_k}}}w_be^{-2\pi i \langle R^{-m_k}b,\lambda\rangle}\right|^2=\|{\bf w}\|^2 \ \forall {\bf w}\in {\mathbb C}^{N^{m_k}}.
\end{equation}
From \eqref{eq3.1a}, we fix $\lambda'\in L^T_k$ and put $w_b = e^{2\pi i \langle R^{-m_k}b,\lambda'\rangle}$, for all $b\in B$. As the term in the sum that corresponds to $\lambda'$ is equal to $N^k$, which is also $\|{\bf w}\|^2$, we obtain that
 $$
\sum_{\lambda\in{L^T_k\setminus\{\lambda'\}}}\left|\sum_{b\in B_{k}}\frac{1}{\sqrt{N^k}}e^{2\pi i \langle R^{-k}b,(\lambda'-\lambda)\rangle}\right|^2= 0.
 $$
This shows that the matrix $\left[e^{2\pi i \langle R^{-k}b,\ell\rangle}\right]_{\ell\in L^T_k, b\in B_n}$ has mutually orthogonal rows and hence  $(R^k, B_k, L^T_k)$ are Hadamard triples. From a similar argument using \eqref{eq3.2a}, we obtain also that $(R^{m_k}, B_{m_k}, \Lambda_k)$ are Hadamard triples.
\end{proof}

\medskip

\section{The case ${\mathcal Z}=\ty$}

In this section, we study the spectral properties of self-affine measures when the set ${\mathcal Z}$ defined in \eqref{eqz} is empty, we prove Theorem \ref{th1.1}. Recall that, for a given expansive integer matrix $R$ and a set $B$ of distinct residue modulo $R{\mathbb Z}^d$, the self-affine measures we are studying satisfy
$$
\mu(E) = \sum_{b\in B}\frac{1}{N}\mu(\tau_b^{-1}(E)).
$$
where $\tau_b(x) = R^{-1}(x+b)$.

Suppose that $(R,B,L)$ forms a Hadamard triple; we have shown that $(R^n,B_n, L^T_n)$ form Hadamard triples. Moreover, if $J_n\equiv L^T_n$ (mod $(R^{T})^n({\mathbb Z}^d)$), then $(R^n,B_n, J_n)$ also forms a Hadamard triple. Now, our goal is to show that some set $\Lambda$ defined as in  \eqref{eqLambda_k} and \eqref{eqlambda} will be a spectrum or frame spectrum for our measure.  Given a sequence of positive integers $n_1,n_2,...$. Recall again that
\begin{equation}\label{eqlambda_Section4}
\Lambda_k = J_{n_1}+ (R^T)^{m_1}J_{n_2}+ (R^T)^{m_2} J_{n_3}+...+ (R^T)^{m_{k-1}} J_{n_{k}},
 \  \Lambda = \bigcup_{k=1}^{\infty}\Lambda_k,
\end{equation}
where $m_k= n_1+...+n_k$, and $0\in J_{n_i}$ for all $i$ is assumed without loss of generality.

\medskip

For the self-affine measure $\mu = \mu(R,B)$,  the Fourier transform can be computed by iterating the invariance identity \eqref{self-affine measure} and we have
$$
\widehat{\mu}(\xi) =  M_B((R^T)^{-1}\xi)\widehat{\mu}((R^T)^{-1}\xi) = ...=\prod_{j=1}^{n}M_B((R^{T})^{-j}\xi)\widehat{\mu}((R^T)^{-n}\xi).
$$
where $M_B(\xi) = \widehat{\delta_{B}}(\xi)= \frac{1}{N}\sum_{b\in B}e^{-2\pi i \langle b,\xi\rangle}$. Note that if $B_n $ is the set in \eqref{B_n},
$$
M_{B_n}(\xi) = \frac{1}{N^n}\sum_{b\in B^n}e^{-2\pi i \langle b,\xi\rangle}= \prod_{j=0}^{n-1}M_{B}((R^{T})^j\xi).
$$
This implies that
\begin{equation}\label{eq4.1}
\widehat{\mu}(\xi) = M_{B_n}((R^{T})^{-n}\xi)\widehat{\mu}((R^T)^{-n}\xi), \ \forall n\ge 1.
\end{equation}
The following lemma is well known and easy to prove.
 \begin{lemma}\label{lem-1}
The following statements are equivalent:
\begin{enumerate}
\item $(R^n,B_n,J_n)$ forms a Hadamard triple.
\item $\delta_{R^{-n}B}$ is a spectral measure with spectrum $J_n$.
\item $M_{B}((R^{T})^{-n} (\lambda-\lambda'))=0$ for all $\lambda\ne\lambda'\in J_n$.
\item For all $\xi\in{\mathbb R}^d$,
$$
\sum_{{\bf \ell}\in J_n}|M_{B_n}(\tau_{\bf \ell}(\xi))|^2\equiv 1,
$$
where $\tau_{\bf \ell}(x) = (R^T)^{-n}(x+{\bf \ell}).$
\end{enumerate}
\end{lemma}

The first three equivalences follow directly from the definitions and the last equivalence follows from the Parseval identity applied to the function $e^{2\pi i \xi x}$ (see e.g. \cite{LaWa02,DJ07d}). We will omit the details of the proof.

 \medskip

 \begin{lemma}\label{lem0}
 Suppose that  $(R^{n_i}, B_{n_i}, J_{n_i})$, $i=1,2,..$, are Hadamard triples. Then for the set $\Lambda$ defined in \eqref{eqlambda_Section4}, the corresponding set of exponential functions $E(\Lambda)$ is a mutually orthogonal set for $\mu$.
 \end{lemma}

 \begin{proof}
 This lemma is indeed a well-known fact  (See e.g.  \cite[Theorem 2.7]{Str00}). Note that $\Lambda_k$ in \eqref{eqlambda_Section4} is an increasing sequence of finite sets, because $0\in J_{n_i}$ for all $i$. Take some distinct $\lambda,\lambda'\in \Lambda$, we choose $k$ so that $\lambda,\lambda'\in \Lambda_k$. By Corollary \ref{corollary2.3}, we know that $(R^{m_k},B_{m_k},\Lambda_k)$ are  Hadamard triples. By Lemma \ref{lem-1}, $M_{B_{m_k}}((R^T)^{-m_k}(\lambda-\lambda'))=0$. Hence, $\widehat{\mu}(\lambda-\lambda')=0$ from  \eqref{eq4.1}.
 \end{proof}

\medskip

We now establish the Fourier frame inequality which implies the completeness of our set of exponentials. The idea is to consider step functions on $T(R,B)$. There is a natural one-to-one correspondence between $B^n$ in Section 2 and $B_n$ in \eqref{B_n}, by identifying ${\bf b}:=(b_0,...,b_{n-1})$ and $b:= \sum_{j=0}^{n-1}R^j b_j$. With an abuse of notation, these two will be used interchangeably. As we are dealing with Hadamard triples, $B$ is a simple digit set for $R$, so that by Theorem \ref{th1.0}, the no-overlap condition is satisfied.

Let ${\mathcal S}_n$ denote the set of all step functions at level $n$ on $T(R,B)$, i.e.,
$$
{\mathcal S}_n = \left\{\sum_{{\bf b}\in B^n}w_{\bf b}{\bf 1}_{T(R,B)_{\bf b}}: w_{\bf b}\in{\mathbb C}\right\}.
$$
Here ${\bf 1}_{T(R,B)_{\bf b}}$ denotes the characteristic function of $T(R,B)_{\bf b}$. It is well known that the set
\begin{equation}\label{eqS}
{\mathcal S} = \bigcup_{n=1}^{\infty} {\mathcal S}_n
\end{equation}
is a dense set of $L^2(\mu)$, but we provide a proof for completeness. Moreover, by iterating the invariance equation $$T(R,B) = \bigcup_{b\in B}\tau_{b}(T(R,B)),$$ it is easy to see that ${\mathcal S}_1\subset{\mathcal S}_2\subset{\mathcal S}_3\subset....$.

\begin{lemma}\label{lem1}
${\mathcal S}$ forms a dense set of $L^2(\mu)$.
Suppose that $f=\sum_{{b}\in B_n}w_{ b}{\bf 1}_{T(R,B)_{\bf b}}\in{\mathcal S}_n$ and $\mu= \mu(R,B)$. Then
\begin{equation}\label{eq4.2}
\int|f|^2d\mu = \frac{1}{N^n}\sum_{{ b}\in B_n}|w_{ b}|^2
\end{equation}
and
\begin{equation}\label{eq4.3}
\int f(x)e^{-2\pi i \xi x}d\mu(x) = \frac{1}{N^n}\widehat{\mu}((R^T)^{-n}\xi)\sum_{{ b}\in B_n}w_{ b}e^{-2\pi i \langle R^{-n}{b}, \xi\rangle}.
\end{equation}
\end{lemma}

\begin{proof}
Take first a continuous function $f$ on $T(R,B)$
and $\epsilon > 0$. Since $T(R,B)$ is compact, the function $f$ is uniformly continuous. We can find $m$ large enough such that the diameter of all sets Ą$\tau_{\bf b}(T(R,B))$, ${\bf b}\in B^m$, is small enough so that $ |f(x) - f(y)| <\epsilon$ for all
$x,y\in\tau_{\bf b}(T(R,B))$ and all ${\bf b} \in B^m$. Consider $g = \sum_{{\bf b}\in B^m}f(\tau_{\bf b}(0)){\bf 1}_{T(R,B)_{\bf b}}$. It is easy to see that $\sup_{x\in T(R,B)}|f(x)-g(x)|<\epsilon$. Hence, ${\mathcal S}$ is uniformly dense in $C(T(R,B))$. As $\mu$ is a regular Borel measure, ${\mathcal S}$ is dense in $L^2(\mu).$

\medskip

The no-overlap condition and the invariance equation for $\mu$ implies that
$$
\mu(T(R,B)_{\bf b}) = \sum_{b'\in B_n}\frac{1}{N^{n}}\mu (\tau_{b'}^{-1}(\tau_{b}(T(R,B)))) = \frac{1}{N^{n}}.
$$
 for all ${\bf b}\in B^n$. This implies \eqref{eq4.2} immediately. To prove \eqref{eq4.3},
\begin{equation}\label{eq4.4.}
\int f(x)e^{-2\pi i \xi x}d\mu(x) = \sum_{{\bf b}\in B^n} w_{{\bf b}}\int_{\tau_{\bf b}(T(R,B))}e^{-2\pi i \langle\xi,x\rangle}d\mu(x).
\end{equation}
Note that
$$
\int_{\tau_{\bf b}(T(R,B))}e^{-2\pi i \langle\xi,x\rangle}d\mu(x) = \frac{1}{N^n}\sum_{{\bf b}'\in B^n}\int{\bf 1}_{\tau_{{\bf b}}(T(R,B))}(\tau_{{\bf b}'}(x))e^{-2\pi i \langle\xi,\tau_{{\bf b}'}(x)\rangle}d\mu(x).
$$
By the no overlap condition, the only non-zero term in the summation above is the one corresponding to  ${\bf b}={\bf b}'$. This implies that
$$
\int_{\tau_{\bf b}(T(R,B))}e^{-2\pi i \langle\xi,x\rangle}d\mu(x) =  \frac{1}{N^n}\int e^{-2\pi i \langle\xi,\tau_{{\bf b}}(x)\rangle}d\mu(x) = \frac{1}{N^n}e^{-2\pi i \langle\xi,R^{-n}{\bf b}\rangle}\widehat{\mu}((R^{T})^{-n}\xi).
$$
Combining this with \eqref{eq4.4.}, we obtain \eqref{eq4.3}.
\end{proof}

\medskip

For the sets $\Lambda_k$ and $\Lambda$ we defined in \eqref{eqlambda_Section4}, we consider the following quantity.
\begin{equation}
\delta (\Lambda) = \inf_{k}\inf_{\lambda\in\Lambda_k}|\widehat{\mu}((R^T)^{-m_k}\lambda)|^2
\label{eqdeltaa}
\end{equation}

\medskip

The following theorem gives a sufficient condition for $\Lambda$ to be a spectrum and this sufficient condition will be realized for some choice of the set $\Lambda$ under the condition ${\mathcal Z}=\ty$.

\begin{theorem}\label{prop_main1}
Suppose that $(R,B,L)$ is a Hadamard triple and let the set $
\Lambda$ be as in \eqref{eqlambda_Section4}. Assume that
\begin{equation}\label{eqdelta}
\delta(\Lambda): = \inf_{k\ge1}\inf_{\lambda\in\Lambda_k}|\widehat{\mu}((R^T)^{-m_k}\lambda)|^2>0
\end{equation}
Then $\Lambda$ is a spectrum for $L^2(\mu(R,B))$.
\end{theorem}

\begin{proof}
We now show the completeness by showing that the following frame bounds hold: for any $f\in L^2(\mu)$,
\begin{equation}\label{eq2.1i}
\delta(\Lambda)\|f\|^2\leq \sum_{\lambda\in\Lambda}\left|\int f(x)e^{-2\pi i \langle\lambda,x\rangle}d\mu(x)\right|^2\leq \|f\|^2.
\end{equation}
The positive lower bound implies the completeness. Now for any $f = \sum_{{\bf b}\in {B}_{m_k}}w_{\bf b}{\bf 1}_{\tau_{\bf b}(R,B)}\in {\mathcal S}$, Lemma \ref{lem1} shows that
\begin{equation}\label{eq3.1}
\int|f|^2d\mu = \frac{1}{N^{m_k}}\sum_{{\bf b}\in {B}_{m_k}}|w_{\bf b}|^2 =  \frac{1}{N^{m_k}}\|{\bf w}\|^2
\end{equation}
where ${\bf w} = (w_{\bf b})_{{\bf b}\in B_{m_k}}$ and
\begin{equation}\label{eq3.2}
\int f(x)e^{-2\pi i \ip{\lambda}{x}}d\mu(x) = \frac{1}{N^{m_k}}\widehat{\mu}((R^T)^{-n}\lambda)\sum_{{\bf b}\in {B}_{m_k}}w_{\bf b} e^{-2\pi i  \ip{R^{-{m_k}}{\bf b}}{ \lambda}}
\end{equation}
which means that
\begin{equation}\label{eq3.3}
\sum_{\lambda\in\Lambda_{k}}\left|\int f(x)e^{-2\pi i \ip{\lambda}{x}}d\mu(x)\right|^2 = \frac{1}{N^{m_k}} \sum_{\lambda\in\Lambda_{k}}|\widehat{\mu}((R^T)^{-{m_k}}\lambda)|^2\left|\sum_{{\bf b}\in {B}_{m_k}}\frac{1}{\sqrt{N^{m_k}}}w_{\bf b}e^{-2\pi i  \ip{R^{-{m_k}}{\bf b}}{\lambda}}\right|^2
\end{equation}
From the definition of $\delta(\Lambda)$, $\delta(\Lambda)\le|\widehat{\mu}((R^T)^{-{m_k}}\lambda)|^2\le 1$ and we thus obtain
$$
 \frac{1}{N^{m_k}}\delta(\Lambda)\|H_{m_k}{\bf w}\|^2\le \sum_{\lambda\in\Lambda_k}\left|\int f(x)e^{-2\pi i \ip{\lambda}{x}}d\mu(x)\right|^2 \le  \frac{1}{N^{m_k}}\|H_{m_k}{\bf w}\|^2,
$$
where $H_{m_k}$ is the Hadamard matrix obtained from the Hadamard triple $(R^{m_k}, B_{m_k}, \Lambda_k)$. So we have $\|H_{m_k}{\bf w}\| = \|{\bf w}\|$ and hence
$$
\delta(\Lambda)\int|f|^2d\mu\le \sum_{\lambda\in\Lambda_k}\left|\int f(x)e^{-2\pi i \ip{\lambda}{x}}d\mu(x)\right|^2 \le \int|f|^2d\mu .
$$
As ${\mathcal S}_{m_k}\subset{\mathcal S}_{m_{\ell}}$ for any $\ell\ge k$, we have
$$
\delta(\Lambda)\int|f|^2d\mu\le \sum_{\lambda\in\Lambda_\ell}\left|\int f(x)e^{-2\pi i \ip{\lambda}{x}}d\mu(x)\right|^2 \le \int|f|^2d\mu,
$$
for all $\ell\ge k$. By letting $\ell$ go to infinity, we have \eqref{eq2.1i}.
\end{proof}

\medskip

\begin{remark} As we will see, actually, in equation \eqref{eq2.1i} we will have indeed the Parseval identity. However, since $\delta(\Lambda)>0$, the frame inequality  is enough to guarantee completeness which, because of the orthogonality, is equivalent to the Parseval identity. In the appendix, we will show directly that Parseval identity holds.
\end{remark}
\medskip

The following proposition shows that some $\Lambda$ will satisfy $\delta(\Lambda)>0$.

\begin{proposition}\label{prop_main2}
Suppose that ${\mathcal Z} = \emptyset$. Then there exists $\Lambda$ built as in \eqref{eqLambda_k} and \eqref{eqlambda} such that $\delta(\Lambda)>0$.
\end{proposition}

\medskip

We now give the proof of this proposition. We start with a lemma.

\begin{lemma}\label{lem2.1}
Suppose that ${\mathcal Z} = \emptyset$ and let $X$ be any compact set on ${\mathbb R}^d$. Then there exist $\epsilon_0>0$, $\delta_0>0$ such that for all $x\in X$, there exists $k_x\in{\mathbb Z}^d$ such that for all $y\in\br^d$ with $\|y\|<\epsilon_0$, we have $|\widehat\mu(x+y+k_x)|^2\geq \delta_0$. In addition, we can choose $k_0=0$ if $0\in X$.
\end{lemma}

\begin{proof}
As ${\mathcal Z} = \emptyset$, for all $x\in X$ there exists $k_x\in\Gamma$ such that $\widehat\mu(x+k_x)\neq 0$. Since $\widehat\mu$ is continuous, there exists an open ball $B(x,\epsilon_x)$ and $\delta_x>0$ such that $|\widehat\mu(y+k_x)|^2\geq\delta_x$ for all $y\in B(x,\epsilon_x)$. Since $X$ is compact, there exist $x_1,\dots,x_m\in X$ such that
$$X\subset\bigcup_{i=1}^mB(x_i,\frac{\epsilon_{x_i}}2).$$
Let $\delta:=\min_i\delta_{x_i}$ and $\epsilon:=\min_i\frac{\epsilon_{x_i}}{2}$. Then, for any $x\in X$, there exists $i$ such that $x\in B(x_i,\frac{\epsilon_{x_i}}2)$. If $\|y\|<\epsilon$, then $x+y\in B(x_i,\epsilon_{x_i})$, so $|\widehat \mu(x+y+k_{x_i})|^2\geq \delta$, and therefore we can redefine $k_x$ to be $k_{x_i}$ to obtain the conclusion. Clearly, we can choose $k_0=0$ if $0\in X$ since $\widehat{\mu}(0)=1$.
\end{proof}

\medskip

\noindent{\it Proof of Proposition \ref{prop_main2}.}  Suppose that $(R,B,L)$ is a Hadamard triple. Then we take $X = T(R^T,L)$, the self-affine set generated by $R^T$ and digit set $L$, which was called the {\it dual fractal} in \cite{JP98}.

Define $J_{n} = L+R^{T}L+...+(R^T)^{n-1}L$.  
 By the definition of self-affine sets, 
$$
(R^{T})^{-(n+p)}J_{n}\subset X , \quad(n\in{\mathbb N},p\geq 0).
$$
Fix the $\epsilon_0$ and $\delta_0$ in Lemma \ref{lem2.1}. We now construct the sets $\Lambda_k$ and $\Lambda$ as in \eqref{eqLambda_k} and \eqref{eqlambda}, by replacing the sets $J_{n_k}$ by some sets $\widehat J_{n_k}$ to guarantee that the number $\delta(\Lambda)$ in \eqref{eqdelta} is positive.

\medskip

 We first start with $\Lambda_{0}: = \{0\} $ and $m_0=n_0=0$. Assuming that $\Lambda_k$ has been constructed, we first choose our $n_{k+1}>n_k$ so that
\begin{equation}\label{eq4.5}
\|(R^T)^{-(n_{k+1}+p)}\lambda\|<\epsilon_0 , \ \forall \ \lambda\in\Lambda_k, p\geq0.
\end{equation}
We then define $m_{k+1} = m_k+n_{k+1}$ and
$$
\Lambda_{k+1} = \Lambda_k+(R^T)^{m_{k}}\widehat{J_{n_{k+1}}}
$$
where
$$
\widehat{J_{n_{k+1}}} = \{j+(R^T)^{n_{k+1}}k(j): j\in J_{n_{k+1}}, \ k(j)\in{\mathbb Z}^d  \}
$$
where $k(j)$ is chosen to be $k_x$ from Lemma \ref{lem2.1}, with $x = (R^{T})^{-n_{k+1}}j  \in X$. As $0\in J_{n_{k}}$ and $k_x=0$ if $x=0$, the sets $\Lambda_k$ are of the form \eqref{eqLambda_k} and  form an increasing sequence. For these sets $\Lambda_k$, we claim that the associated $\Lambda$ in \eqref{eqlambda} satisfies $\delta(\Lambda)>0$.

\medskip

To justify the claim, we note that if $\lambda\in \Lambda_{k}$, then
$$
\lambda = \lambda'+(R^T)^{m_{k-1}}j  +(R^T)^{m_{k}}k(j),
$$
where $\lambda'\in \Lambda_{k-1}$, $j\in J_{n_{k}}$. This means that
$$
(R^T)^{-m_k}\lambda = (R^T)^{-m_k}\lambda'+(R^T)^{-n_k}j+ k(j).
$$
By \eqref{eq4.5}, $\|(R^T)^{-m_k}\lambda'\|<\epsilon_0$. From Lemma \ref{lem2.1}, since $(R^T)^{-n_k}j\in X$, we must have $|\widehat{\mu}((R^T)^{-m_k}\lambda)|^2\geq\delta_0>0$. As $\delta_0$ is independent of $k$, the claim is justified and hence this completes the proof of the proposition.
\qquad$\Box$.

Combining Theorem \ref{prop_main1} and Proposition \ref{prop_main2}, we settle the case ${\mathcal Z}=\emptyset$.

\begin{proof}[Proof of Theorem \ref{th1.1}]
 To prove Theorem \ref{th1.1}, suppose first that ${\mathcal Z}=\emptyset$. We take the sets $(J_{n_i})$ in Proposition \ref{prop_main2} so that $(R^{n_i},B_{n_i},J_{n_i})$ are Hadamard triples and $\delta(\Lambda)>0$. Then, $\Lambda$ is a spectrum for $\mu(R,B)$. It is clearly a subset of ${\mathbb Z}^d$ since all sets $J_{n_i}$ are so. Hence, $\mu(R,B)$ is a spectral measure with a spectrum in ${\mathbb Z}^d$.

  \medskip

  Conversely, if ${\mathcal Z}\neq\emptyset$, then there exists $\xi_0\in{\mathcal Z}$ such that ${\widehat{\mu}}(\xi_0+k)=0$ for all $k\in{\mathbb Z}^d$. Denote $e_{\xi}(x) = e^{2\pi i \langle\xi,x\rangle}$. We have
$$\langle e_{\xi_0},e_k\rangle = 0, \ \forall \ k\in{\mathbb Z}^d
$$
This means that the exponentials $E(\Lambda)$ cannot be complete in $L^2(\mu)$ whenever $\Lambda$ is a subset of ${\mathbb Z}^d$. Hence, there is no spectrum in ${\mathbb Z}^d$ for $\mu$.
\end{proof}

\medskip

\section{A further reduction}

Let $R$ be an expansive matrix with integer entries on ${\mathbb R}^d$. An $R$-invariant lattice is a lattice $\Lambda$ such that $R(\Lambda)\subset \Lambda$.  We define ${\mathbb Z}[R,B] $ to be the smallest $R$-invariant lattice containing all $B_n =  B+ RB+...+R^{n-1}B$. In this section, our goal is to show that proving our main Theorem (Theorem \ref{thmain}), it suffices to show it for the case ${\mathbb Z}[R,B] = {\mathbb Z}^d$ (See Proposition \ref{przrb}).
\medskip

\begin{definition}\label{defconj}
Let $R_1,R_2$ be $d\times d$ integer matrices, and the finite sets $B_1,B_2,L_1,L_2$ be in $\bz^d$.
We say that two triples $(R_1, B_1, L_1)$ and $(R_2, B_2, L_2)$ are {\it conjugate} (through the matrix $M$) if there exists an integer matrix $M$ such that $R_2 = MR_1M^{-1}$, $B_2 = MB_1$ and
$L_2 = (M^T)^{-1}L_1$.
\end{definition}

\medskip

\begin{proposition}\label{prconj}
Suppose that $(R_1, B_1, L_1)$ and $(R_2, B_2, L_2)$ are two conjugate triples, through the matrix $M$. Then
\medskip

(i) If $(R_1,B_1,L_1)$ is a Hadamard triple then so is $(R_2,B_2,L_2)$.

\medskip

(ii)The measure $\mu(R_1,B_1)$ is spectral with spectrum $\Lambda$ if and only if $\mu(R_2,B_2)$ is spectral with spectrum $(M^T)^{-1}\Lambda$.
\end{proposition}

\begin{proof}
The proof follows from a simple computation, see e.g. \cite[Proposition 3.4]{DJ07d}.
\end{proof}

\medskip

\begin{proposition}\label{przrb}
\begin{itemize}
\item If the lattice $\bz[R,B]$ is not full-rank, then the dimension can be reduced. More precisely, there exists $1\leq r<d$ and a unimodular matrix $M\in GL(n,\bz)$ such that $M(B)\subset \bz^r\times\{0\}$ and
\begin{equation}
MRM^{-1}=\begin{bmatrix}
A_1& C\\
0& A_2
\end{bmatrix}
\label{eqzrb1}
\end{equation}
where $A_1\in M_r(\bz)$, $C\in M_{r,d-r}(\bz), A_2\in M_{d-r}(\bz)$. In addition,  $M(T(R,B))\subset \br^r\times \{0\}$ and the Hadamard triple $(R,B,L)$ is conjugate to the Hadamard triple $(MRM^{-1},MB,(M^T)^{-1}L)$, which is a triple of lower dimension.

\medskip

\item If the lattice $\bz[R,B]$ is full rank but not $\bz^d$, then  the system $(R,B,L)$ is conjugate to one $(\tilde R,\tilde B,\tilde L)$ for which $\bz[\tilde R,\tilde B]=\bz^d$. Moreover, the conjugation matrix $M$ can be chosen such that ${\mathbb Z}[R,B] = M({\mathbb Z}^d)$. 
    \end{itemize}
\end{proposition}

\begin{proof}
If the lattice $\bz[R,B]$ is not full-rank, then it spans a proper rational subspace (i.e., having a basis with rational components) $V$ of $\br^d$ of dimension $r$. Since $\bz[R,B]$ is invariant under $R$, it follows that $RV\subset V$ and since $R$ is invertible, the dimensions must match so $RV=V$.  Then there is a unimodular matrix $M\in GL(n,\bz)$ that maps $V$ into the first $r$ coordinate axes, that is $MV=\br^r\times\{0\}$, see e.g. \cite[Theorem 4.1 and Corollary 4.3b] {Sch86}. Then also $MB\subset \br^r\times\{0\}$. Since
$$T(R,B)=\left\{\sum_{n=1}^\infty R^{-n}b_n : b_n\in B\mbox{ for all }b\in B\right\},$$
we get that $T(R,B)$ is in $V$ so $MT(R,B)\subset \br^r\times\{0\}$.

\medskip

The subspace $\br^r\times \{0\}$ is invariant for $MRM^{-1}$ and this implies that $M$ has the form in \eqref{eqzrb1}. Since $M$ is unimodular $M^{-1}$ is also an integer matrix so $MRM^{-1}$ is an integer matrix. The other statements follow by a simple computation.

\medskip

If $\bz[R,B]$ is full rank but not $\bz^d$ then $\bz[R,B]=M\bz^d$ for some invertible integer matrix $M$. If $\{e_j\}$ are the canonical vectors in $\br^d$, then $RMe_j\in\bz[R,B]$ so $RMe_j=M\tilde r_j$ for some $r_j\in\bz^d$. So $RM=M\tilde R$ for an integer matrix $\tilde R$, i.e., $\tilde R=M^{-1}RM$. Since $B\subset\bz[R,B]=M\bz^d$, there exists $\tilde B$ in $\bz^d$ such that $B=M\tilde B$ so $\tilde B=M^{-1}B$.
We have $M^{-1}R^kB=\tilde R^k\tilde B$ so $\bz[\tilde R,\tilde B]=M^{-1}\bz[R,B]=\bz^d$. The other statements follow from an easy computation.
\end{proof}

\medskip

We now  provide a proof of Theorem \ref{thmain} in the case of dimension one, giving us another proof among others from the literature \cite{LaWa02,DJ06}. By rescaling, there is no loss of generality if we assume that $\gcd(B)=1$. Note also that $\gcd(B)=1$ is equivalent to $\bz[R,B]=\bz$.

\begin{theorem}\label{example5.0}
Suppose that $R$ is an integer and $(R,B,L)$ forms a Hadamard triple in  ${\mathbb R}^1$ with $\gcd(B)=1$. Then the associated self-similar measure $\mu(R,B)$ satisfies ${\mathcal Z} = \emptyset$, with $\mathcal Z$ defined in \eqref{eqz}, and is spectral with a spectrum in ${\mathbb Z}$.
\end{theorem}

\begin{proof}
We can assume $0\in B$. Suppose that ${\mathcal Z}\neq \emptyset$. As $\widehat{\mu}(0)=1$, ${\mathcal Z}\cap {\mathbb Z}=\emptyset$. Then we pick $\xi_0\in{\mathcal Z}$ and $\xi_0\not\in{\mathbb Z}$. We claim the following fact is true: For any $\ell\in L$,
\begin{equation}\label{eq5.1}
 M_B(\tau_{\ell}(\xi_0))\neq 0,  \ \Rightarrow \ \tau_{\ell}(\xi_0)\in {\mathcal Z},
  \end{equation}
	where $\tau_{\ell}(x)=R^{-1}(x+\ell)$, $x\in\br$.

Indeed, by considering $k$ of the form $ \ell+R e$ and $e\in{\mathbb Z}$, we have
$$
\begin{aligned}
0=\widehat{\mu}(\xi_0+k) =& M_B(R^{-1}(\xi_0+\ell+R e))\widehat{\mu}(R^{-1}(\xi_0+\ell+R e)) \\
=&  M_B(\tau_{\ell}(\xi_0))\widehat{\mu}(\tau_{\ell}(\xi_0)+ e)\\
\end{aligned}
$$
As $M_B(\tau_{\ell}(\xi_0))\neq 0 $, we must have $\widehat{\mu}(\tau_{\ell}(\xi_0)+ e)=0$ for all $e\in\bz$ and hence $\tau_{\ell}(\xi_0)\in{\mathcal Z}$. With this fact in mind, we define $Y_0=\{\xi_0\}$ and define inductively the set $Y_n$ by
$$
Y_n = \{\tau_{\ell}(\xi): \ell\in L, \ \xi\in Y_{n-1}, \ M_{B}(\tau_{\ell}(\xi))\neq 0\}.
$$
By \eqref{eq5.1}, $Y_n\subset{\mathcal Z}$ and $Y_n\cap {\mathbb Z}=\emptyset$. From the fact that $(R,B,L)$ is a Hadamard triple and Lemma \ref{lem-1}, we have
\begin{equation}\label{eqPa}
\sum_{\ell\in L}|M_B(\tau_{\ell}(\xi))|^2\equiv 1.
\end{equation}
This means that all the sets $Y_n$ are non-empty.  Also if $\xi_n\in Y_n$, then $\xi_n= \tau_{\ell_n}\circ...\circ\tau_{\ell_1}(\xi_0) = R^{-n}(\xi_0+\ell_1+...+R^{n-1}\ell_n)$. This means $|\xi_n|\leq |\xi_0|+D$, where $D = \mbox{diam}(T(R,L))$. Hence, $\sup_n \{|\xi_n|:\xi_n\in Y_n\}$ is bounded. We also notice that for different $\ell_0\ell_1\dots \ell_n\neq \ell_0'\ell_1'\dots \ell_n'$ the corresponding $\xi_n$ and $\xi_n'$ are different, since $L$ is a simple digit set for $R$. Therefore the cardinality of $Y_n$ is increasing.

\medskip

 On ${\mathbb R}^{1}$, $\widehat{\mu}$ has only finitely many zeros in a bounded set. Therefore, there exists $n_0$ such that for all $n\ge n_0$, the cardinality of $Y_n$ becomes a constant. This means that when $n\geq n_0$, each $\xi_n$ has only one offspring $\xi_{n+1} =\tau_{\ell_0}(\xi_n)$, i.e , there is only one $l_0\in L$ such that $M_B(\tau_{l_0}(\xi_n))\neq 0$ and so $M_B(\tau_{\ell}(\xi_n))=0$ for all $\ell_n\neq0$. From  \eqref{eqPa}, $|M_B(\tau_{\ell_0}(\xi_n)) |= \left|\frac{1}{N}\sum_{b\in B}e^{2\pi i b\tau_\ell(\xi_n)}\right|=1$. This implies we have equality in a triangle inequality, and since $0\in B$, we get that $b\tau_{\ell_0}(\xi_n)\in {\mathbb Z}$ for all $b\in B$. As gcd$(B)=1$, we can take $m_b\in{\mathbb Z}$ such $\sum_{b\in B} bm_b=1$ and this forces $\tau_{\ell_0}(\xi_n) = \sum_{b\in B} m_b (b\tau_{\ell_0}(\xi_n))\in{\mathbb Z}$. This is a contradiction, since $\tau_{\ell_0}(\xi_n)\in\mathcal Z$ and $\mathcal Z\cap \bz=\ty$.
\end{proof}
\medskip

Thus, in dimension one $\bz[R,B]=\bz$ implies that $\mathcal Z=\ty$. In the end of this section, we show that the implication is no longer true in higher dimensions.  We illustrate this possibility, when ${\mathbb Z}[R,B] = {\mathbb Z}^2$ and ${\mathcal Z}\ne \emptyset$, with a simple example.

\begin{example}\label{example5.3}
Let $R = \left[
           \begin{array}{cc}
             4 & 0 \\
             1 & 2 \\
           \end{array}
         \right]
$, $$B= \left\{\left[
                \begin{array}{c}
                  0 \\
                  0 \\
                \end{array}
              \right], \left[
                \begin{array}{c}
                  0 \\
                  3\\
                \end{array}
              \right], \left[
                \begin{array}{c}
                  1 \\
                  0 \\
                \end{array}
              \right], \left[
                \begin{array}{c}
                  1 \\
                  3 \\
                \end{array}
              \right]
\right\} \ \mbox{and} \
 L= \left\{\left[
                \begin{array}{c}
                  0 \\
                  0 \\
                \end{array}
              \right], \left[
                \begin{array}{c}
                  2 \\
                  0\\
                \end{array}
              \right],  \left[
                \begin{array}{c}
                  0 \\
                  1 \\
                \end{array}
              \right], \left[
                \begin{array}{c}
                  2 \\
                  1 \\
                \end{array}
              \right]
\right\}.$$ Then $(R,B,L)$ forms a  Hadamard triple and ${\mathbb Z}[R,B] = {\mathbb Z}^2$. However, the set  defined in \eqref{eqz} ${\mathcal Z}\neq \emptyset$ for the measure $\mu = \mu(R,B)$.
\end{example}

\begin{proof}
It is a direct check to see $(R,B,L)$ forms a  Hadamard triple. As $R\left[
                \begin{array}{c}
                  1 \\
                  0 \\
                \end{array}
              \right] = \left[
                \begin{array}{c}
                  4 \\
                  1 \\
                \end{array}
              \right]$, and the vectors $\left[
                \begin{array}{c}
                  4 \\
                  1 \\
                \end{array}
              \right]$,$\left[
                \begin{array}{c}
                  1 \\
                  0 \\
                \end{array}
              \right]$ generate ${\mathbb Z}^2$, we have that  ${\mathbb Z}[R,B] = {\mathbb Z}^2$.  As $$M_B(\xi_1,\xi_2) = \frac{1}{4}(1+e^{2\pi i \xi_1})(1+e^{2\pi i 3\xi_2}),$$ it follows that the zero set of $M_B$, denoted by $Z(M_B)$, is equal to
$$
Z(M_B) = \left\{\left[
                \begin{array}{c}
                 \frac12+n \\
                  y \\
                \end{array}
              \right]: n\in{\mathbb Z}, y\in {\mathbb R}\right\}\cup \left\{\left[
                \begin{array}{c}
                  x \\
                  \frac16+\frac13 n  \\
                \end{array}
              \right]:x\in {\mathbb R}, n\in{\mathbb Z}\right\}.
$$
Let $(R^T)^j = \left[
           \begin{array}{cc}
             4^j & a_j \\
              0& 2^j \\
           \end{array}
         \right]$, for some $a_j\in{\mathbb Z}$. As $\widehat{\mu}(\xi) = \prod_{j=1}^{\infty}M_B((R^T)^{-j}(\xi))$, the zero set of $\widehat{\mu}$, denoted by $Z(\widehat{\mu})$, is equal to
$$\begin{aligned}
Z(\widehat{\mu}) =& \bigcup_{j=1}^{\infty}(R^{T})^{j} Z(M_B) \nonumber\\
=&\bigcup_{j=1}^{\infty}\left\{\left[
                \begin{array}{c}
                 4^j(\frac12+n)+a_jy \\
                  2^jy \\
                \end{array}
              \right]: n\in{\mathbb Z}, y\in {\mathbb R}\right\}\cup \left\{\left[
                \begin{array}{c}
                  4^jx+a_j(\frac16+\frac13 n) \\
                  2^{j}\left(\frac16+\frac13 n\right)  \\
                \end{array}
              \right]:x\in {\mathbb R}, n\in{\mathbb Z}\right\}.
\end{aligned}
$$
We claim that the points in  $\left[
                \begin{array}{c}
                 0 \\
                  \frac13 \\
                \end{array}
              \right]+{\mathbb Z}^2$ are in $Z(\widehat{\mu})$ which shows ${\mathcal Z}\neq \emptyset$. Indeed, for any $\left[
                \begin{array}{c}
                 m \\
                  \frac13+n \\
                \end{array}
              \right]$, $m,n\in{\mathbb Z}$, we can write it as $\left[
                \begin{array}{c}
                 m \\
                  \frac{1+3n}{3} \\
                \end{array}
              \right]$. We now rewrite the second term in the union in $Z(\widehat{\mu})$ as ${\mathbb R}\times \{\frac{2^{j-1}(1+2n)}{3}\}$. As any integer can be written as $2^j p$, for some $j\ge 0$ and odd number $p$, this means that $\left[
                \begin{array}{c}
                 m \\
                  \frac{1+3n}{3} \\
                \end{array}
              \right]\in Z(\widehat{\mu})$, justifying the claim. As ${\mathcal Z}\neq\emptyset,$ this shows that there is no spectrum in ${\mathbb Z}^2$ for this measure.

					The measure $\mu(R,B)$ is spectral. In fact,
              $$
T(R,B) = \bigcup_{x\in K_1}\{x\}\times ([0,3]+g(x)),
$$
where $K_1$ is the Cantor set of $1/4$ contraction ratio and digit $\{0,1\}$ and $g:[0,1]\rightarrow{\mathbb R}$ is a measurable function obtained from the off-diagonal entries of powers of the inverse of $R$.


We can see that $\mu(R,B)$ in the previous example is indeed spectral since it is the one-fourth Cantor set on the $x$-direction and it has a bunch of translated fibers equal to the interval $[0,3]$ on the $y$-direction. We can form $\Lambda\times (\frac13{\mathbb Z})$ as our spectrum, where $\Lambda$ is any spectra for the one-fourth Cantor set. In fact, we will show that this is always the situation in general case when ${\mathcal Z}\ne \emptyset$ and ${\mathbb Z}[R,B] = {\mathbb Z}^d$

\end{proof}

\medskip
\section{Periodic Invariant sets of a dynamical system}
In this section, we introduce a dynamical system associated to ${\mathcal Z}$ following some techniques from \cite{CCR}. By analyzing  this dynamical system, we can reduce the digit set $B$ to a quasi-product form, as we will see in the next section.

Using Propositions \ref{prconj} and \ref{przrb}, we can make the following assumption:

\medskip

{\bf Assumption.} We assume in the sequel that $\bz[R,B]=\bz^d$.

\medskip

We start with the following definition.

\begin{definition}\label{definv}
Let $u\geq0$ be an entire function on $\br^d$, i.e., real analytic on $\br^d$. Let $\cj L$ be a simple digit set for $R^T$. Suppose that
\begin{equation}
\sum_{\ell\in \cj L}u((R^T)^{-1}(x+\ell))>0,\quad(x\in\br^d)
\label{eqdefinv1}
\end{equation}

Define the maps
$$\tau_{\ell}(x)=(R^T)^{-1}(x+\ell),\quad (\ell\in\cj L,x\in\br^d).$$

A closed set $K$ in $\br^d$ is called {\it $u$-invariant (with respect to the system $(u, R^T,\cj L)$)} if, for all $x\in K$ and all $\ell\in\cj L$
$$
 u\left(\tau_{\ell}(x)\right)>0 \ \Longrightarrow \ \tau_{\ell}(x)\in K.
 $$
We say that the transition, using $\ell$, from $x$ to $\tau_{\ell}(x)$ is possible, if $\ell\in \cj L$ and $u\left((R^{T})^{-1}(x+\ell)\right)>0$. We say that $K$ is ${\mathbb Z}^d$-periodic if $K+n=K$ for all $n\in{\mathbb Z}^d$.
\end{definition}

\medskip

We say that a (vector) subspace $W$ of ${\mathbb R}^d$ is a {\it rational subspace} if $W$ has a basis of vectors with rational components. The following theorem follows from Proposition 2.5, Theorem 2.8 and Theorem 3.3 in \cite{CCR}.
\begin{theorem}\label{thccr}\cite{CCR}
Let $\cj L$ be a complete set of representatives $(\mod R^T(\bz^d))$.
Let $u\geq0$ be an entire function on $\br^d$ and let $K$ be  a closed $u$-invariant $\bz^d$-periodic set different from $\br^d$. Suppose in addition that $g$ is an entire function which is zero on $K$. Then

\medskip

(i) there exists a point $x_0\in\br^d$, such that $(R^T)^mx_0\equiv x_0(\mod \ \bz^d)$ for some integer $m\geq1$, and

 \medskip

 (ii) a proper rational subspace $W$ (may equal $\{0\}$)  such that $R^{T}(W) = W$ and the union
$$
\mathcal S=\bigcup_{k=0}^{m-1}((R^T)^kx_0+W+\bz^d)
$$
is invariant and $g$ is zero on $\mathcal S$.

Moreover, all possible transitions from a point in $(R^T)^kx_0+W+\bz^d$, $1\leq k\leq m$, lead to a point in $(R^T)^{k-1}x_0+W+\bz^d$.
\end{theorem}

\medskip

Let $(R,B,L)$ be a Hadamard triple and we aim to apply Theorem \ref{thccr} to our set ${\mathcal Z}$ in \eqref{eqz}. We define the function

\begin{equation}
u_B(x)=\left|\frac{1}{N}\sum_{b\in B}e^{2\pi i\ip{b}{x}}\right|^2,\quad(x\in\br^d).
\label{eq1.12.1}
\end{equation}
Recall that, taking the Fourier transform of the invariance equation \eqref{self-affine}, we can compute explicitly the Fourier transform of $\mu = \mu(R,B)$ as
\begin{equation}
|\widehat\mu(\xi)|^2=u_B((R^T)^{-1}\xi)|\widehat\mu((R^T)^{-1}(\xi))|^2,\quad(x\in\br^d).
\label{eq1.12.2}
\end{equation}
Iterating \eqref{eq1.12.2}, we obtain
\begin{equation}
|\widehat\mu(x)|^2=\prod_{n=1}^\infty u_B((R^T)^{-n}x),\quad(x\in\br^d),
\label{eq1.12.3}
\end{equation}
and the convergence in the product is uniform on compact sets. See e.g. \cite{DJ07d}. It is well known that both $u_B$ and $|\widehat{\mu}|^2$ are entire functions on ${\mathbb R}^d$.

\begin{proposition}\label{pr1.13}
Suppose that $(R,B,L)$ forms a Hadamard triple
and $\bz[R,B]=\bz^d$. Let $\cj L$ be a complete set of representatives $(\mod R^T(\bz^d))$ containing $L$. Suppose that the set
$$\mathcal Z:=\left\{\xi\in \br^d : \widehat\mu (\xi+k)=0\mbox{ for all }k\in\bz^d\right\}$$
is non-empty. Then

 \medskip

 (i) $\mathcal Z$ is $u_B$-invariant.

 \medskip
 (ii) There exist a point $x_0\in\br^d$ such that $(R^T)^mx_0\equiv x_0(\mod R^T(\bz^d))$, for some integer $m\geq 1$.

  \medskip
  (iii) There exists a proper rational subspace $W\neq\{0\}$ of $\br^d$ such that $R^{T}(W)=W$ and the union
$$\mathcal S=\bigcup_{k=0}^{m-1}((R^T)^kx_0+W+\bz^d)$$
is $u_B$-invariant and is contained in $\mathcal Z$.

\medskip
Moreover, all possible transitions  from a point in $(R^T)^kx_0+W+\bz^d$, $1\leq k\leq m$, lead to a point in $(R^T)^{k-1}x_0+W+\bz^d$.

\end{proposition}

\begin{proof}

We first prove  that $\mathcal Z$ is $u_B$-invariant. Take $x\in\mathcal Z$ and $\ell\in\cj L$ such that $u_B((R^T)^{-1}(x+\ell))>0$. Let $k\in\bz^d$. We have, with \eqref{eq1.12.2},
$$
\begin{aligned}
0=|\widehat\mu(x+\ell+R^Tk)|^2=&
u_B((R^T)^{-1}(x+\ell+R^Tk))|\widehat\mu((R^T)^{-1}(x+\ell+R^Tk))|^2\\
=&u_B((R^T)^{-1}(x+\ell))|\widehat\mu((R^T)^{-1}(x+\ell)+k)|^2.\\
\end{aligned}
$$
Therefore, $\widehat\mu((R^T)^{-1}(x+\ell)+k)=0$ for all $k\in\bz^d$. So $(R^T)^{-1}(x+\ell)$ is in $\mathcal Z$, and this shows that $\mathcal Z$ is $u_B$-invariant.

\medskip

Since $(R,B,L)$ form a Hadamard triple, by the Parseval identity in Lemma \ref{lem-1}(iv),
\begin{equation}\label{eqqmf}
\sum_{\ell\in L}u_B((R^T)^{-1}(x+\ell))=1,\quad(x\in\br^d),
\end{equation}
Hence,
\begin{equation}
\sum_{\ell\in \cj L}u_B((R^T)^{-1}(x+\ell))>0,\quad(x\in\br^d).
\label{eqqmf2}
\end{equation}

We can apply Theorem \ref{thccr} with $u=u_B$ and $g=\widehat\mu$ to obtain all other conclusions except the non-triviality of $W$. We now check that $W\neq\{0\}$. Suppose $W=\{0\}$. First we show that for $1\leq k\leq m$ there is a unique $\ell\in \cj L$ such that
$u_B((R^T)^{-1}((R^T)^kx_0+\ell)>0$. Equation \eqref{eqqmf2} shows that there exists at least one such $\ell$. Assume that we have two different $\ell$ and $\ell'$ in $\cj L$ with this property. Then the transitions are possible, so
\begin{equation}\label{eq1.12}
(R^T)^{-1}((R^T)^kx_0+\ell)\equiv (R^T)^{k-1}x_0\equiv (R^T)^{-1}((R^T)^kx_0+\ell') \ (\mod  \ (\bz^d)).
\end{equation}
But then $\ell\equiv \ell'(\mod R^T(\bz^d))$ and this is impossible since $\cj L$ is a complete set of representatives.

\medskip

Recall that we assume that $0\in B$. From \eqref{eqqmf}, and since the elements in $L$ are distinct $(\mod R^T(\bz^d))$, we see that there is exactly one $\ell_k\in L$ such that $u_B((R^T)^{-1}((R^T)^kx_0+\ell_k))>0$. Therefore $u_B((R^T)^{-1}((R^T)^kx_0+\ell_k))=1$. But then, by \eqref{eq1.12}, $u_B((R^T)^{k-1}x_0)=1$. We have
$$
\left|\sum_{b\in B}e^{2\pi i \ip{b}{(R^{T})^{k-1}x_0}}\right| = N.
$$
 As $\#B = N$ and $0\in B$, we have equality in the triangle inequality, and we get that $e^{2\pi i \ip{b}{(R^T)^{k-1}x_0}}=1$ for all $b\in B$. Then $\ip{R^{k-1}b}{x_0}\in{\mathbb Z}$ for all $b\in B$, $1\leq k\leq m$. Because $(R^T)^m x_0\equiv x_0(\mod\bz^d)$, we get that $\ip{R^kb}{x_0}\in{\mathbb Z}$ for all $k\geq 0$ and thus
$$x_0\in\bz[R,B]^\perp:=\{x\in\br^d : \ip{\lambda}{x}\in\bz\mbox{ for all }\lambda\in \bz[R,B]\}.$$ Since $\bz[R,B]=\bz^d$, this means that $x_0\in \bz^d$. But $x_0\in\mathcal Z$, so $1=\widehat\mu(0)=\widehat\mu(x_0-x_0)=0$, which is a contradiction. This shows $W\neq\{0\}$.
\end{proof}

\medskip

Because $W\neq \{0\}$, we will see that we can conjugate $R$ through some $M$ so that, after conjugation, $(R,B,L)$ has a upper triangular structure.

\begin{proposition}\label{pr1.14}
Suppose that $(R,B,L)$ forms a Hadamard triple and ${\mathbb Z}[R,B] = {\mathbb Z}^d$ and let $\mu = \mu(R,B)$ be the associated self-affine measure $\mu=\mu(R,B)$. Suppose that the set
$$
\mathcal Z:=\left\{x\in \br^d : \widehat\mu (x+k)=0\mbox{ for all }k\in\bz^d\right\},
$$
is non-empty. Then there exists an integer unimodular matrix $M$ such that the following assertions hold:

\begin{enumerate}
	\item The matrix $\tilde R:=MRM^{-1}$ is of the form
	\begin{equation}
	\tilde R=\begin{bmatrix} R_1&0\\ C &R_2\end{bmatrix},
	\label{eq1.14.1}
	\end{equation}
	with $R_1\in M_r(\bz)$, $R_2\in M_{d-r}(\bz)$ expansive integer matrices and  $C\in M_{(d-r)\times r}(\bz)$.
	\item If $\tilde B=MB$ and $\tilde L=(M^T)^{-1}L$, then $(\tilde R, \tilde B,\tilde L)$ is a Hadamard triple.
	\item The measure $\mu(R,B)$ is spectral with spectrum $\Lambda$ if and only if the measure $\mu(\tilde R,\tilde B)$ is spectral with spectrum $(M^T)^{-1}\Lambda$.
	\item There exists $y_0\in\br^{d-r}$ such that $(R_2^T)^my_0\equiv y_0(\mod R_2^T(\bz^d))$ for some integer $m\geq 1$ such that the union
	$$\tilde {\mathcal S}=\bigcup_{k=0}^{m-1}(\br^r\times \{(R_2^T)^ky_0\}+\bz^d)$$
	is contained in the set
	$$\tilde{\mathcal Z}:=\left\{x\in \br^d : \widehat{\tilde \mu} (x+k)=0\mbox{ for all }k\in\bz^d\right\},$$
	where $\tilde \mu=\mu(\tilde R,\tilde B)$. The set $\tilde{\mathcal S}$ is invariant (with respect to the system $(u_{\tilde B},\tilde R^T,\tilde {\cj L})$, where $\tilde {\cj L}$ is a complete set of representatives $(\mod \tilde R^T\bz^d)$. In addition, all possible transitions from a point in $\br^r\times \{(R_2^T)^ky_0\}+\bz^d$, $1\leq k\leq m$ lead to a point in $\br^r\times \{(R_2^T)^{k-1}y_0\}+\bz^d$.
\end{enumerate}
\end{proposition}

\begin{proof}
We use Proposition \ref{pr1.13} and we have $x_0$ and a rational subspace $W\neq\{0\}$ invariant for $R$ with all properties in Proposition \ref{pr1.13}. By \cite[Theorem 4.1 and Corollary 4.3b]{Sch86}, there exists a unimodular matrix $M$ such that $MW=\br^r\times \{0\}$. Under the conjugation with the matrix $M$, the matrix $\tilde R$ will have the form (\ref{eq1.14.1}) since $W$ is an invariant subspace.  properties (ii)-(iii) follows from Proposition \ref{prconj}. Finally, properties (iv) follows from Proposition \ref{pr1.13}(iii) by identifying $W$ as $\br^r\times \{0\}$ and taking $y_0$ as the second component of $Mx_0$.

\end{proof}

\medskip

\section{The quasi-product form}

Using Proposition \ref{pr1.14}, we can replace $(R,B,L)$ by $(\tilde R,\tilde B, \tilde L)$ and make the assumptions that $R$ has the form in \eqref{eq1.14.1}, so
\begin{equation}
R=\begin{bmatrix}
R_1& 0\\ C& R_2
\end{bmatrix},
\label{eqi}
\end{equation}
and it satisfies the property (iv) in Proposition \ref{pr1.14}:

\medskip
 {\bf Assumption (iv).}
There exists $y_0\in\br^{d-r}$ such that $(R_2^T)^my_0\equiv y_0(\mod \ R_2^T(\bz^d))$ for some integer $m\geq 1$ and such that the union
	$$ {\mathcal S}=\bigcup_{k=0}^{m-1}(\br^r\times \{(R_2^T)^ky_0\}+\bz^d)$$
	is contained in the set
	$${\mathcal Z}:=\left\{x\in \br^d : \widehat{ \mu} (x+k)=0\mbox{ for all }k\in\bz^d\right\},$$
	where $ \mu=\mu( R, B)$. The set ${\mathcal S}$ is invariant (with respect to the system $(u_{ B}, R^T, {\cj L})$, where $ {\cj L}$ is a complete set of representatives $(\mod  R^T(\bz^d))$. In addition, all possible transitions from a point in $\br^r\times \{(R_2^T)^ky_0\}+\bz^d$, $1\leq k\leq m$ lead to a point in $\br^r\times \{(R_2^T)^{k-1}y_0\}+\bz^d$.

\medskip

In this section, we will prove that, in our case, which is ${\mathcal Z}\neq \emptyset$, the Hadamard triple is conjugate to one that has a quasi-product form.

\medskip

We use $A\times B$ to denote the Cartesian product of $A$ and $B$ so that $A\times B = \{(a,b):a\in A, b\in B\}$. We first introduce the following notations.

\begin{definition}\label{def1.15}
For a vector $x\in \br^d$, we write it as $x=(x^{(1)},x^{(2)})^T$ with $x^{(1)}\in \br^r$ and $x^{(2)}\in\br^{d-r}$. We denote by $\pi_1(x)=x^{(1)}$, $\pi_2(x)=x^{(2)}$. For a subset $A$ of $\br^d$, and $x_1\in\br^r$, $x_2\in \br^{d-r}$, we denote by
$$A_2({x_1}):=\{y\in\br^{d-r} :(x_1,y)^T\in A\},\quad A_1({x_2}):=\{x\in\br^r : (x,x_2)^T\in A\}.$$
\end{definition}

   Our main theorem in this section is as follows:

\begin{theorem}\label{th_quasi}
Suppose that $$R=\begin{bmatrix}
R_1& 0\\ C& R_2
\end{bmatrix},
$$$(R,B,L_0)$ is a Hadamard triple that satisfies the Assumption (iv), and let $\mu = \mu(R,B)$ be the associated self-affine measure and ${\mathcal Z}\neq \emptyset$. Then the set $B$ has the following quasi-product form:
\begin{equation}
B=\left\{(u_i,v_i+Qc_{i,j})^T : 1\leq i\leq N_1, 1\leq j\leq|\det R_2|\right\},
\label{eq1.18.1}
\end{equation}
where
 \begin{enumerate}
 \item $N_1 = N/|\det R_2|$, \item $Q$ is a $(d-r)\times (d-r)$ integer matrix with $|\det Q|\geq 2$ and $R_2Q=Q\tilde R_2$ for some $(d-r)\times(d-r)$ integer matrix $\widetilde{R_2}$,
     \item the set $\{Qc_{i,j}: 1\leq j\leq |\det R_2|\}$ is a complete set of representatives $(\mod R_2(\bz^{d-r}))$, for all $1\leq i\leq N_1$.
\end{enumerate}
\medskip

Moreover, one can find some set $L\equiv L_0 (\mod \ R^T(\mathbb Z^d))$ so that $(R,B,L)$ is a Hadamard triple and $(R_1,\pi_1(B), L_1(\ell_2))$ and $(R_2,B_2(b_1),\pi_2(L))$ are Hadamard triples on ${\mathbb R}^r$ and ${\mathbb R}^{d-r}$ respectively, for all $\ell_2\in\pi_2(L)$ and all $b_1\in\pi_1(B)$.
\end{theorem}

In establishing Theorem \ref{th_quasi}, we need a series of lemmas. We will assume

\medskip

\mbox{\bf (A1)}: \ \mbox{ Hadamard triple $(R,B,L)$ satisfies \eqref{eqi} and Assumption (iv)}

\medskip

 First, the following lemma allows us to replace $L$ by an equivalent set $L'$ with certain injectivity property.

\begin{lemma}\label{lem1.15-}
Suppose that {\bf (A1)} holds. Then there exists set $L'$ and a complete set of representatives $\cj L'$ such that $(R,B,L')$ is a Hadamard triple and the following property holds:
\begin{equation}\label{eq1.15.0}
\ell,\ell'\in L' \ \mbox{(or $\in \cj L'$) and} \ \pi_2(\ell)\equiv \pi_2(\ell') \ \left(\mod \  R_2^T(\bz^{d-r})\right) \ \ \Longrightarrow \  \ \pi_2(\ell)=\pi_2(\ell').
\end{equation}
\end{lemma}

\begin{proof}
 From $L$ (or ${\cj L}$), we define an equivalence relation  by
 $$
 \ell\sim\ell' \ \mbox{if and only if} \  \pi_2(\ell)\equiv \pi_2(\ell') \ \left(\mod \  R_2^T(\bz^{d-r})\right).
  $$
  From each of the equivalence class, we fix an $\ell=(\ell_1,\ell_2)^T\in L$. Suppose that  $\ell'=(\ell_1',\ell_2')^T \sim \ell=(\ell_1,\ell_2)^T$ but $\ell_2\ne \ell_2'$. We replace $\ell'$ by another representative
$$
\ell''=\ell'+R^T(0,(R_2^T)^{-1}(\ell_2-\ell_2'))^T\in\bz^d.
$$
 Then $\ell''\equiv \ell' (\mod R^T({\bz^{d}}))$ and
 $\pi_2(\ell'')=\ell_2= \pi_2(\ell).$  Define $L'$ to be the set of elements with all these replacements. Then $L\equiv L'$ $(\mod R^T({\bz^{d}}))$ and $(R,B,L')$ is a Hadamard triple since $\ip{R^{-1}b}{R^{T}m}$ $\in{\mathbb Z}$ for any $m\in{\mathbb Z}^d$. Furthermore, for all elements in the same equivalence class, $\pi_2(\ell)$ are all equal. \eqref{eq1.15.0} is thus satisfied  and we obtain our lemma.
\end{proof}

\medskip

To simplify the notation, in what follows we relabel $L'$ by $L$ and $\cj L'$ by $\cj L$ so that $L$ and $\cj L$ have the property \eqref{eq1.15.0}. We write it as follows:
\medskip

\mbox{\bf (A2)}: \  $L,{\cj L}$ \ \mbox{satisfy}:
$$
\ell,\ell'\in L \ \mbox{(or $\in \cj L$) and} \ \pi_2(\ell)\equiv \pi_2(\ell') \ \left(\mod \  R_2^T(\bz^{d-r})\right) \ \ \Longrightarrow \  \ \pi_2(\ell)=\pi_2(\ell').
$$
\medskip

\begin{lemma}\label{lem1.15}
Suppose that \mbox{\bf (A1)} and \mbox{\bf (A2)} holds. Then
\begin{enumerate}
%
%
\item For every $b_1\in \pi_1(B)$ and $b_2\neq b_2'$ in $B_2(b_1)$,
\begin{equation}
\sum_{\ell_2\in \pi_2(L)}\# L_1(\ell_2)e^{2\pi i \ip{R_2^{-1}(b_2-b_2')}{\ell_2}}=0.
\label{eq1.15.1}
\end{equation}
Also, for all $b_1\in \pi_1(B)$,  $\#B_2(b_1)\leq \#\pi_2(L)$  and the elements in $B_2(b_1)$ are not congruent $\mod R_2(\bz^{d-r})$.

\medskip

\item For every $\ell_2\in \pi_2(L)$ and $\ell_1\neq \ell_1'$ in $L_1(\ell_2)$,
\begin{equation}
\sum_{b_1\in\pi_1(B)}\# B_2(b_1)e^{2\pi i \ip{R_1^{-1}b_1}{(\ell_1-\ell_1')}}=0.
\label{eq1.15.2}
\end{equation}
Also, for all $\ell_2\in\pi_2(L)$, $\#L_1(\ell_2)\leq \#\pi_1(B)$ and the elements in $L_1(\ell_2)$ are not congruent $\left(\mod \  R_1^T(\bz^{r})\right)$.

\medskip
\item The set $\pi_2(\cj L)$ is a complete set of representatives $(\mod  \ R_2^T(\bz^{d-r}))$ and, for every $\ell_2\in \pi_2(\cj L)$, the set $\cj L_1(\ell_2)$ is a complete set of representatives $(\mod \  R_1^T(\bz^r))$.
\end{enumerate}
\end{lemma}

\begin{proof}
We prove (i).
Take $b_1\in\pi_1(B)$ and $b_2\neq b_2'$ in $B_2(b_1)$, from the mutual orthogonality, we have
$$
\begin{aligned}
0=&\sum_{\ell_2\in\pi_2(L)}\sum_{\ell_1\in L_1(\ell_2)}e^{2\pi i \ip{R^{-1}(0,b_2-b_2')^T}{ (\ell_1,\ell_2)^T}}=\sum_{\ell_2\in\pi_2(L)}\sum_{\ell_1\in L_1(\ell_2)}e^{2\pi i \ip{R_2^{-1}(b_2-b_2')}{ \ell_2}}\\
=&\sum_{\ell_2\in\pi_2(L)}\#L_1(\ell_2)e^{2\pi i \ip{R_2)^{-1}(b_2-b_2')}{\ell_2}}.
\end{aligned}
$$
This shows \eqref{eq1.15.1} and that the rows of the matrix
$$
\left(\sqrt{\# L_1(\ell_2)}e^{2\pi i \ip{(R_2^T)^{-1}b_2}{\ell_2}}\right)_{\ell_2\in \pi_2(L),b_2\in B_2(b_1)}
$$
are orthogonal. Therefore $\#B_2(b_1)\leq \#\pi_2(L)$ for all $b_1\in \pi_1(B)$. Equation \eqref{eq1.15.1} implies that the elements in $B_2(b_1)$ cannot be congruent $\mod R_2^T(\bz^{d-r})$

\medskip

(ii) follows from an analogous computation as in (i).

\medskip

For (iii), the elements in $\pi_2(\cj L)$ are not congruent $(\mod R_2^T(\bz^{d-r}))$, since \eqref{eq1.15.0} is satisfied. If $\ell_2\in \pi_2(\cj L)$ and $\ell_1,\ell_1'\in \cj L_1(l_2)$ are congruent $(\mod R_1^T(\bz^r))$ then $(\ell_1,\ell_2)^T\equiv (\ell_1',\ell_2)^T$$(\mod \ R^T(\bz^d))$. Thus, $\ell_1=\ell_1'$, as $\cj L$ is a complete set of representatives of $R^T({\mathbb Z}^d)$. From these, we have $\#\pi_2(L)\leq |\det R_2|$ and, for all $\ell_2\in \pi_2(\cj L)$, $\#L_1(\ell_2)\leq |\det R_1|$. Since
$$|\det R|=|\det R_1||\det R_2|\geq \sum_{\ell_2\in \pi_2(\cj L)}\#\cj L_1(\ell_2)=\#\cj L=|\det R|,$$
we must have equalities in all inequalities and we get that the sets are indeed {\it complete} sets of representatives.
\end{proof}

\medskip

\begin{lemma}\label{lem1.16}
Suppose that \mbox{\bf (A1)} and \mbox{\bf (A2)} holds. Let $1\leq j\leq m$. If the  the transition from $(x,(R_2^T)^jy_0)^T$ is possible with the digit $\ell\in \cj L$, then $\pi_2(\ell)=0$.
\end{lemma}

\begin{proof}
If the transition is possible with digit $\ell=(\ell_1,\ell_2)^T$, then, by
Assumption (iv),  $$(R^T)^{-1}((x,(R_2^T)^jy_0)^T+(\ell_1,\ell_2)^T)\equiv (y,(R_2^T)^{j-1}y_0)^T(\mod\bz^d),$$ for some $y\in\br^r$, and therefore $(R_2^T)^{-1}\ell_2\equiv 0 \ (\mod \ \bz^{d-r})$, so $\ell_2\equiv 0 \ (\mod R_2^T(\bz^{d-r}))$. By \mbox{\bf (A2)}, $\ell_2=0$.
\end{proof}

\medskip

\begin{lemma}
Suppose that \mbox{\bf (A1)} and \mbox{\bf (A2)} holds. Let $y_j:=(R_2^T)^jy_0$, $1\leq j\leq m$. Then, for all $x\in\br^r$ and all $\ell=(\ell_1,\ell_2)\in \cj L$ with $\pi_2(\ell)=\ell_2\neq 0$, we have that, for all $b_1\in\pi_1(B)$,
\begin{equation}
\sum_{b_2\in B_2(b_1)}e^{2\pi i \ip{b_2}{ (R_2^T)^{-1}(y_j+ \ell_2)}}=0.
\label{eq1.17.1}
\end{equation}
\end{lemma}
\begin{proof}
We have that $(R^T)^{-1}$ is of the form $$(R^T)^{-1}=\begin{bmatrix}
(R_1^T)^{-1}& D\\ 0& (R_2^T)^{-1}
\end{bmatrix}.
$$
Then, for all $x\in\br^r$ and all $\ell=(\ell_1,\ell_2)\in \cj L$ with $\pi_2(\ell)=\ell_2\neq 0$, we have that $u_B((R^T)^{-1}((x,y_j)^T+(\ell_1,\ell_2)^T))=0$, because such transitions are not possible by Lemma \ref{lem1.16}. Then
$$
\sum_{b_1\in \pi_1(B)}\sum_{b_2\in B_2(b_1)}e^{2\pi i (\ip{b_1}{ (R_1^T)^{-1}(x+l_1)+D(y_1+l_2)}+\ip{b_2}{ (R_2^T)^{-1}(y_j+l_2)})}=0.
$$
Since $x$ is arbitrary,  it follows that
$$
\sum_{b_1\in\pi_1(B)}e^{2\pi i \ip{b_1}{x}}\sum_{b_2\in B_2(b_1)}e^{2\pi i \ip{b_2}{(R_2^T)^{-1}(y_j+l_2)}}=0 \ \mbox{for all}  \ x\in{\mathbb R}^d.
$$
Therefore, by linear independence of exponential functions, we obtain \eqref{eq1.17.1}.
\end{proof}

\medskip

\begin{lemma}\label{lem1.18}
Suppose that \mbox{\bf (A1)} and \mbox{\bf (A2)} holds. For every $b_1\in\pi_1(B)$, the set $B_2(b_1)$ is a complete set of representatives $(\mod R_2(\bz^{d-r}))$. Therefore $\#B_2(b_1)=|\det R_2|=\#\pi_2(L)$ and
$(R_2,B_2(b_1),\pi_2(L))$ is a Hadamard triple. Also, for every $\ell_2\in \pi_2(L)$, $\#L_1(\ell_2)=\#\pi_1(B)=\frac{N}{|\det R_2|}=:N_1$ and
$(R_1,\pi_1(B),L_1(\ell_2))$ is a Hadamard triple. 
\end{lemma}

\begin{proof}
Let $b_1\in\pi_1(B)$.  We know from Lemma \ref{lem1.15}(i) that the elements of $B_2(b_1)$ are not congruent $(\mod \ R_2(\bz^{d-r}))$. We can identify $B_2(b_1)$ with a subset of the group $\bz^{d-r}/R_2(\bz^{d-r})$. The dual group is $\bz^{d-r}/R_2^T(\bz^{d-r})$ which we can identify with $\pi_2(\cj L)$.
For a function $f$ on $\bz^{d-r}/R_2(\bz^{d-r})$, the Fourier transform is
$$\hat f(\ell_2)=\frac{1}{\sqrt{|\det R_2|}}\sum_{b_2\in \bz^{d-r}/R_2(\bz^{d-r})} f(b_2)e^{-2\pi i\ip{b_2}{(R_2^T)^{-1}\ell_2}},\quad (\ell_2\in \pi_2(\cj L)).$$

Let $1\leq j\leq m$. Consider the function
\begin{equation}\label{f(b)}
f(b_2)=\left\{\begin{array}{cc}e^{-2\pi i\ip{b_2}{(R_2^T)^{-1}y_j}},&\mbox{ if }b_2\in B_2(b_1)\\
0,&\mbox{ if }b_2\in (\bz^{d-r}/R_2(\bz^{d-r}))\setminus B_2(b_1).\end{array}\right.
\end{equation}
Then equation \eqref{eq1.17.1} shows that $\hat f(\ell_2)=0$ for $\ell_2\in \pi_2(\cj L)$, $\ell_2\neq 0$. Thus $\hat f=c\cdot\chi_{\{0\}}$ for some constant $c$ and by $\hat{f}(0) = c$,
\begin{equation}\label{c}
c=\frac{1}{\sqrt{|\det R_2|}}\sum_{b_2\in B_2(b_1)}e^{-2\pi i \ip{b_2}{(R_2^T)^{-1}(y_j)}}.
\end{equation}

\medskip

 Now we apply the inverse Fourier transform and we get
$$f(b_2)=\frac{1}{\sqrt{|\det R_2|}}\sum_{\ell_2\in \pi_2(\cj L)}c\chi_{\{0\}}(\ell_2)e^{2\pi i \ip{b_2}{ (R_2^T)^{-1}\ell_2}}=\frac{c}{\sqrt{|\det R_2|}}.$$
So $f(b_2)$ is constant and therefore $B_2(b_1)=\bz^{d-r}/R_2(\bz^{d-r})$, which means that $B_2(b_1)$ is a complete set of representatives and $\#B_2(b_1)=|\det R_2|$. Since the elements in $\#\pi_2(L)$ are not congruent $(\mod R_2^T(\bz^{d-r}))$, we get that $\# \pi_2(L)\leq |\det R_2|$, and with Lemma \ref{lem1.15} (i), it follows that $\#\pi_2(L)=\#B_2(b_1)=|\det R_2|$. In particular, $\pi_2(L)$ is a complete set of representatives $(\mod R_2^T(\bz^{d-r}))$, so $(R_2,B_2(b_1),\pi_1(L))$ form a Hadamard triple.

\medskip

Since $\sum_{b_1\in \pi_1(B)}\#B_2(b_1)=N$, we get that $\#\pi_1(B)=N/|\det R_2|$. With Lemma \ref{lem1.15}(ii), we have
$$N=\sum_{\ell_2\in\pi_2(L)}\#L_1(\ell_2)\leq \#\pi_2(L)\pi_1(B)=N.$$
Therefore we have equality in all inequalities so $\#L_1(\ell_2)=\#\pi_1(B)=N/|\det R_2|$. Then \eqref{eq1.15.2} shows that $(R_1,\pi_1(B),L_1(\ell_2))$ is a Hadamard triple for all $\ell_2\in \pi_2(L)$.
\end{proof}

\medskip

\begin{proof}[Proof of Theorem \ref{th_quasi}] Using Lemma \ref{lem1.15-}, we can find $L$ so that the Hadamard triple  $(R,B,L)$ satisfies both {\bf (A1)} and {\bf (A2)}. By Lemma \ref{lem1.18}, we know that $B$ must have the form
 $$
 \bigcup_{b_1\in\pi_1(B)} \{b_1\}\times B_2(b_1)
 $$
 where $\#\pi_1(B) = N_1$ and $B_2(b_1)$ is a set of complete representative (mod $R_2^T({\mathbb Z}^{d-r})$). By enumerating elements $\pi_1(B)  = \{u_1,...,u_{N_1}\}$ and $B_2(u_{i}) = \{d_{i,1},...,d_{i,|\det R_2|}\}$, we can write
 $$
 B =\left\{(u_i,d_{i,j})^T : 1\leq i\leq N_1, 1\leq j\leq|\det R_2|\right\}.
 $$
It suffices to show $d_{i,j}$ are given by $v_i+Qc_{i,j}$ where $Q$ has the properties (ii) and (iii) in the theorem. From the equation \eqref{f(b)} and the fact that $f$ is a constant,
we have, for $b_2\in B_2(b_1)$ and $b_1\in\pi_1(B)$,
$$
e^{-2\pi i \ip{b_2}{(R_2^T)^{-1}y_j}}=f(b_2)=\frac{c}{\sqrt{|\det R_2|}},
$$
which implies (from \eqref{c}) that
$$
\frac{1}{|\det R_2|}\sum_{b_2'\in B_2(b_1)}e^{2\pi i \ip{(b_2-b_2')}{(R_2^T)^{-1}y_j}}=1.
$$
By applying the triangle inequality to the sum above, we see that we must have
$$
e^{2\pi i\ip{b_2-b_2'}{(R_2^T)^{-1}y_j}}=1,
$$
which means
\begin{equation}\label{b_2-b_2'}
\ip{b_2-b_2'}{(R_2^T)^{-1} y_j}\in\bz \  \mbox{for all} \ b_2,b_2'\in B_2(b_1), b_1\in \pi_1(B), 1\leq j\leq m.
\end{equation} Here we recall that  $y_j = (R_2^T)^{j}y_0$.

\medskip

Define now the lattice
$$\Gamma:=\{x\in\bz^{d-r} : \ip{x}{(R_2^T)^{-1}y_j}\in{\mathbb Z},  \ \forall \ 1\leq j\leq m\}.$$

\medskip

We first claim that the lattice $\Gamma$ is of full-rank. Indeed, since $(R_2^T)^my_j\equiv y_j(\mod \bz^{d-r})$, it follows that all the points $(R_2^T)^{-1}(y_j)$ have only rational components. Let $\tilde m$ be a common multiple for all the denominators of all the components of the vectors $(R_2^T)^{-1}(y_j)$, $1\leq j\leq m$. If $\{e_i\}$ are the canonical vectors in $\br^{d-r}$, then $\ip{\tilde me_i}{ (R_2^T)^{-1}(y_j)}\in\bz$ so $\tilde m e_i\in\Gamma$, and thus $\Gamma$ is full-rank.

\medskip

Next we prove that $\Gamma$ is a proper sublattice of $\bz^{d-r}$. The vectors $y_j$ are not in $\bz^{d-r}$ because the points $(0,y_j)$ are contained in $\mathcal Z$, by Assumption (iv), so $\widehat\mu((0,y_j)^T+k)=0$ for all $k\in\bz^d$, and that would contradict the fact that $\widehat\mu(0)=1$. This implies that the vectors $(R_2^T)^{-1}(y_j)$ are not in $\bz^{d-r}$ so $\Gamma$ is a proper sublattice of $\bz^{d-r}$.

\medskip

Since $\Gamma$ is a full-rank lattice in $\bz^{d-r}$, there exists an invertible matrix with integer entries $Q$ such that $\Gamma=Q(\bz^{d-r})$, and since $\Gamma$ is a proper sublattice, it follows that $|\det Q|>1$ so $|\det Q|\geq 2$.
In addition, we know from \eqref{b_2-b_2'} that, for all $u_i\in \pi_1(B)$ and $d_{i,j},d_{i,j'}\in B_2(u_i)$, $d_{i,j}-d_{i,j'}\in\Gamma$. Therefore, if we fix an element $v_i=d_{i,j_0}\in B_2(a_i)$, then all the elements in $B_2(a_i)$ are of the form $d_{i,j}=v_i+Qc_{i,j}$ for some $c_{i,j}\in\bz^{d-r}$. The fact that $B_2(a_i)$ is a complete set of representatives $(\mod R_2(\bz^{d-r}))$ (Lemma \ref{lem1.18}) implies that  the set of the corresponding elements $Qc_{i,j}$ is also a complete set of representatives $(\mod R_2(\bz^{d-r}))$. This shows (iii).

\medskip

It remains to show $R_2Q = Q\widetilde{R_2}$ for some for some $(d-r)\times(d-r)$ integer matrix $\widetilde{R_2}$.  Indeed, if $x\in\Gamma$, and $0\leq j\leq m-1$, then
$\ip{R_2x}{(R_2^T)^jy_0}=\ip{x}{(R_2^T)^{j+1}y_0}\in\bz$, since $(R_2^T)^my_0\equiv y_0 \ (\mod\bz^{d-r})$. So $R_2x\in\Gamma$. Then, for the canonical vectors $e_i$, there exist $\tilde e_i\in\bz^{d-r}$ such that $R_2Qe_i=Q\widetilde{e_i}$. Let $\widetilde{R_2}$ be the matrix with columns $\widetilde{e_i}$. Then $R_2Q=Q\widetilde{R_2}$.

\medskip

Finally, since {\bf (A2)} is satisfied, the Hadamard triple properties of both  $(R_1,\pi_1(B), L_1(\ell_2))$ and $(R_2,B_2(b_1),\pi_1(L))$ on ${\mathbb R}^r$ and ${\mathbb R}^{d-r}$ respectively are direct consequences of Lemma \ref{lem1.18}.
\end{proof}
\medskip

\section{Proof of the theorem}
In this section, we will prove our main theorem. We first need to study the spectral property of Hadamard triples that are in the quasi-product form.
Suppose now the pair $(R,B)$ is in the quasi-product form
\begin{equation}\label{R_4.1}R=\begin{bmatrix}
R_1&0\\
C&R_2
\end{bmatrix}\end{equation}
\begin{equation}
B=\left\{(u_i,d_{i,j})^T : 1\leq i\leq N_1, 1\leq j\leq N_2:=|\det R_2|\right\},
\label{eq1.23.1}
\end{equation}
and $\{d_{i,j}:1\leq j\leq N_2\}$  is a complete set of representatives $(\mod R_2\bz^{d-r})$ (in fact, $d_{i,j} = v_i +Qc_{i,j}$ as in Theorem \ref{th_quasi}). We will show that the measure $\mu=\mu(R,B)$ has a quasi-product structure.

\medskip

Note that we have
$$R^{-1}=\begin{bmatrix}
R_1^{-1}&0\\
-R_2^{-1}CR_1^{-1}&R_2^{-1}
\end{bmatrix}$$
and, by induction,
$$R^{-k}= \begin{bmatrix}
R_1^{-k}&0\\
D_k&R_2^{-k}
\end{bmatrix},\mbox{ where }D_k:=-\sum_{l=0}^{k-1}R_2^{-(l+1)}CR_1^{-(k-l)}.$$
 For the invariant set $T(R,B)$, we can express it as a set of infinite sums,
 $$
 T(R,B)=\left\{\sum_{k=1}^\infty R^{-k}b_k : b_k\in B\right\}.
 $$
Therefore any element $(x,y)^T\in T(R,B)$ can be written in the following form
$$x=\sum_{k=1}^\infty R_1^{-k}u_{i_k}, \quad y=\sum_{k=1}^\infty D_ku_{i_k}+\sum_{k=1}^\infty R_2^{-k}d_{i_k,j_k}.$$
Let $X_1$ be the attractor (in $\br^r)$ associated to the IFS defined by the pair $(R_1,\pi_1(B)=\{u_i :1\leq i\leq N_1\})$ (i.e. $X_1=T(R_1,\pi_1(B))$). Let $\mu_1$ be the (equal-weight) invariant measure associated to this pair.

\medskip

For each sequence $\omega=(i_1i_2\dots)\in\{1,\dots,N_1\}^{\bn} = \{1,\dots, N_1\}\times\{1,\dots, N_1\}\times...$, define
\begin{equation}\label{eqxomega}
x(\omega)=\sum_{k=1}^\infty R_1^{-k}u_{i_k}.
\end{equation}

 By Lemma \ref{lem1.18},  $(R_1,\pi_1(B))$ forms Hadamard triple with some $L_1(\ell_2)$. Thus, the measure $\mu(R_1,\pi_1(B))$ has the no-overlap property by Theorem \ref{th1.0}. It implies that there is a bijective correspondence, up to measure zero, between the set $\Omega_1:=\{1,\dots,N_1\}^{\bn}$ and $X_1$ (See \cite{Ki} for detail). Hence,  for $\mu_1$-a.e. $x\in X_1$, there is a unique $\omega$ such that $x(\omega)=x$. Because of the correspondence, we define the $\omega$ as $\omega(x)$. Furthermore, the measure $\mu_1$ on $X_1$ is the pull-back of the product measure on $\Omega_1$ which assigns equal probabilities $\frac1{N_1}$ to each digit.

\medskip

For $\omega=(i_1i_2\dots)$ in $\Omega_1$, define
$$\Omega_2(\omega):=\{(d_{i_1,j_1}d_{i_2,j_2}\dots d_{i_n,j_n}\dots) : j_k\in \{1,\dots,N_2\}\mbox{ for all }k\in\bn\}.$$
For $\omega\in\Omega_1$, define $g(\omega):=\sum_{k=1}^\infty D_ka_{i_k}$ and $g(x):=g(\omega(x))$, for $x\in X_1$.  Also $\Omega_2(x):=\Omega_2(\omega(x))$.
For $x\in X_1$, define
$$X_2(x):=X_2(\omega(x)):=\left\{\sum_{k=1}^\infty R_2^{-k}d_{i_k,j_k}: j_k\in\{1,\dots,N_2\}\mbox{ for all }k\in\bn\right\}.$$
Note that the attractor $T(R,B)$ has the following form

$$T(R,B)=\{(x,g(x)+y)^T: x\in X_1,y\in X_2(x)\}.$$

\medskip

For $\omega\in \Omega_1$,  consider the product probability measure $\mu_\omega$, on $\Omega_2(\omega)$, which assigns equal probabilities $\frac{1}{N_2}$ to each digit $d_{i_k,j_k}$ at level $k$.
Next, we define the measure $\mu_\omega^2$ on $X_2(\omega)$. Let
$r_\omega:\Omega_2(\omega)\rightarrow X_2(\omega)$,
$$r_\omega(d_{i_1,j_1}d_{i_2,j_2}\dots)=\sum_{k=1}^\infty R_2^{-k}d_{i_k,j_k}.$$
Define $\mu_x^2:=\mu_{\omega(x)}^2:=\mu_{\omega(x)}\circ r_{\omega(x)}^{-1}$.

\medskip

Note that the measure $\mu_x^2$ is the infinite convolution product $\delta_{R_2^{-1}B_2(i_1)}\ast\delta_{R_2^{-2}B_2(i_2)}\ast\dots$, where $\omega(x)=(i_1i_2\dots)$, $B_2(i_k):=\{d_{i_k,j} : 1\leq j\leq N_2\}$ and $\delta_A:=\frac{1}{\#A}\sum_{a\in A}\delta_a$, for a subset $A$ of $\br^{d-r}$. The following lemmas were proved in \cite{DJ07d}.

\begin{lemma}\label{lem1.23}\cite[Lemma 4.4]{DJ07d}
For any bounded Borel functions on $\br^d$,
$$\int_{T(R,B)}f\,d\mu=\int_{X_1}\int_{X_2(x)}f(x,y+g(x))\,d\mu_x^2(y)\,d\mu_1(x).$$
\end{lemma}

\begin{lemma}\label{lem1.24}\cite[Lemma 4.5]{DJ07d}
If $\Lambda_1$ is a spectrum for the measure $\mu_1$, then
$$F(y):=\sum_{\lambda_1\in\Lambda_1}|\widehat\mu(x+\lambda_1,y)|^2=\int_{X_1}|\widehat\mu_s^2(y)|^2\,d\mu_1(s),\quad(x\in\br^r,y\in\br^{d-r}).$$
\end{lemma}

\medskip

We recall also the Jorgensen-Pedersen Lemma for checking when a countable set is a spectrum for a measure.

\begin{lemma}\label{lemJP}\cite{JP98}
Let $\mu$ be a compactly supported probability measure. Then $\Lambda$ is a spectrum for $L^2(\mu)$ if and only if
$$
\sum_{\lambda\in\Lambda}|\widehat{\mu}(\xi+\lambda)|^2\equiv 1.
$$
\end{lemma}

We need the following key proposition.

\begin{proposition}\label{lem1.25}
For the quasi-product form given in \eqref{R_4.1} and \eqref{eq1.23.1}, there exists a lattice $\Gamma_2$ such that for $\mu_1$-almost every $x\in X_1$, the set $\Gamma_2$ is a spectrum for the measure $\mu_x^2$.
\end{proposition}

\begin{proof}
First we replace the first component $(R_1,\pi_1(B))$ by a more convenient pair which allows us to use the theory of self-affine tiles from \cite{LW2}. Define
$$R^\dagger:=\begin{bmatrix}
N_1&0\\
0&R_2
\end{bmatrix},$$
$$B^\dagger=\left\{(i,d_{i,j})^T: 1\leq i\leq N_1-1, 1\leq j\leq |\det R_2|\right\}.$$
We will use the super-script $\dagger$ to refer to the pair $(R^\dagger,B^\dagger)$.

\medskip

As $d_{i,j}$ is a complete residue ($\mod R_2(\bz^{d-r})$), the set $B^\dagger$ is a complete set of representatives $(\mod R^\dagger(\bz^{d-r+1}))$. By \cite[Theorem 1.1]{LW2}, $\mu^\dagger$ is the normalized Lebesgue measure on $T(R^\dagger,B^\dagger)$ and
this tiles $\br^{d-r+1}$ with some lattice $\Gamma^{\ast}\subset{\mathbb Z}^{d-r+1}$. The attractor $X_1^{\dagger}$ corresponds to the pair $(N_1,\{0,1,\dots,N_1-1\})$ so $X_1^\dagger$ is $[0,1]$ and $\mu_1^\dagger$ is the Lebesgue measure on $[0,1]$.  We make the following claim:

 \medskip

 {\it Claim:} the set $T(R^\dagger, B^\dagger)$  tiles $\br^{d-r+1}$ with a set of the form $\bz\times \tilde{\Gamma}_2$, where $\tilde\Gamma_2$ is a lattice in $\br^{d-r}$.

 \medskip

\noindent {\it Proof of claim:} This claim was established implicitly in the proof of Theorem 1.1, in section 7, p.101 of \cite{LW2}, but we present a proof here for completeness. Let $\Gamma^{\ast}$ be the lattice on ${\mathbb R}^{d-r+1}$ which is a tiling set of $T(R^\dagger, B^\dagger)$. We observe that the orthogonal projection of  $T(R^\dagger, B^\dagger)$ onto the first coordinate is $[0,1]$. Hence, for any $\gamma\in\Gamma^{\ast}$, the orthogonal projection of $T(R^\dagger, B^\dagger)+\gamma$ is $[0,1]+\gamma_1$, where $\gamma= (\gamma_1,\gamma_2)^T$. As $\Gamma^{\ast}\subset{\mathbb Z}^{d-r+1}$, the projections $[0,1]+\gamma_1$ are measure disjoint for different $\gamma_1$'s. Therefore, the tiling of $T(R^\dagger, B^\dagger)$ by $\Gamma^{\ast}$ naturally divides up into cylinders:
$$
U(\gamma_1): = ([0,1]+\gamma_1)\times {\mathbb R}^{d-r}.
$$
Focusing on one of the cylinders, say $U(0)$, this cylinder is tiled by $\tilde{\Gamma}$ where
$$
\tilde{\Gamma} = \Gamma\cap (\{0\}\times {\mathbb Z}^{d-r}).
$$
As ${\mathbb R}^{d-r+1} = \bigcup_{\gamma_1}U(\gamma_1)$, this means $T(R^\dagger, B^\dagger)$  also  tiles by ${\mathbb Z}\times\tilde{\Gamma}_2$, where $\tilde\Gamma_2$ is the second component of $\tilde\Gamma$. This completes the proof of the claim.
\medskip

 Because of the claim, it follows from the well-known result of Fuglede \cite{Fug74} that $\mu^\dagger$ has a spectrum of the form $\bz\times\Gamma_2$, with $\Gamma_2$ the dual lattice of $\tilde\Gamma_2$.

\medskip

We prove that $\Gamma_2$ is an orthogonal set for the measure ${\mu^\dagger}_s^2$ for $\mu_1^\dagger$-almost every $s\in X_1^\dagger$. Indeed, for $\gamma_2\neq 0$  in $\Gamma_2$, since $\bz\times\Gamma_2$ is a spectrum for $\mu^\dagger$, we have for all $\lambda_1\in\bz$, with Lemma \ref{lem1.23},
$$0=\int_{T(R^\dagger,B^\dagger)}e^{-2\pi i \ip{(\lambda_1,\gamma_2)}{(x,y)}}\,d\mu^\dagger(x,y)=\int_{X_1^\dagger}\int_{{X^\dagger}_2(x)}e^{-2\pi i\ip{(\lambda_1,\gamma_2)}{(x,y)}}\,d{\mu^\dagger}_x^2(y)\,d\mu_1^\dagger(x)$$
$$=\int_{X_1^\dagger}e^{-2\pi i\lambda_1x}\int_{{X^\dagger}_2(x)}e^{-2\pi i\ip{\gamma_2}{ y}}\,d{\mu^\dagger}_x^2(y)\,d\mu^\dagger_1(x).$$
This implies that that
$$\int_{{X^\dagger}_2(x)}e^{-2\pi i \ip{\gamma_2}{y}}\,d{\mu^\dagger}_x^2(y)=0,$$
for all $\gamma_2\in\Gamma_2\setminus\{0\}$ for $\mu^\dagger_1$-a.e. $x\in X_1^\dagger$. This means that $\Gamma_2$ is an orthogonal sequence for ${\mu^\dagger}_x^2$ for $\mu^\dagger_1$-a.e. $x\in X_1^\dagger$ so
\begin{equation}
\sum_{\gamma_2\in\Gamma_2}|\widehat{{\mu^\dagger}_x^2}(y+\gamma_2)|^2\leq 1,\quad(y\in\br^{d-r}),
\label{eq1.25.1}
\end{equation}
for $\mu^\dagger_1$-a.e. $x\in X_1^\dagger$.
With Lemma \ref{lem1.24}, we have
$$1=\sum_{\gamma_2\in\Gamma_2}\sum_{\lambda_1\in\bz}|\widehat\mu^\dagger(x+\lambda_1,y+\gamma_2)|^2=\int_{X^\dagger_1}\sum_{\gamma_2\in\Gamma_2}|\widehat{{\mu^\dagger}_s^2}(y+\gamma_2)|^2\,d\mu^\dagger_1(s).$$
With \eqref{eq1.25.1}, we have
\begin{equation}
\sum_{\gamma_2\in\Gamma_2}|\widehat{{\mu^\dagger}_s^2}(y+\gamma_2)|^2= 1,\quad(y\in\br^{d-r}),
\label{eq1.25.2}
\end{equation}
for $\mu_1^\dagger$-a.e. $s\in X_1^\dagger$,
which means that $\Gamma_2$ is a spectrum for almost every measure ${\mu^\dagger}_s^2$ by Lemma \ref{lemJP}.

\medskip

Now, we are switching back to our original pair $(R,B)$. Note that we have the maps $x:\Omega_1\rightarrow X_1$ and $x^\dagger:\Omega_1\rightarrow X_1^\dagger$, defined by $\omega\mapsto x(\omega)$ as above in \eqref{eqxomega}, and analogously for $x^\dagger$. The maps are measure preserving bijections. Let $\Psi:X_1\rightarrow X_1^\dagger$ be the composition $\psi=x^\dagger\circ x^{-1}$. i.e.
 $$
 \Psi\left(\sum_{j=1}^{\infty} R_1^{-j}u_j\right) = \sum_{j=1}^{\infty} N_1^{-j}j.
 $$
 Consider the measure $\nu (E) = \mu_1^{\dagger}(\Psi(E))$ for Borel set $E$ in $T(R_1,\pi_1(B))$. Because of the no-overlap condition, we can check easily that $\nu$ and $\mu_1$ agree on all the cylinder sets of $T(R_1,\pi_1(B))$. i.e.
 $$
 \nu(\tau_{i_1}\circ...\circ\tau_{i_n}(T(R_1,\pi_1(B)))) = \frac{1}{N^n} =  \mu_1(\tau_{i_1}\circ...\circ\tau_{i_n}(T(R_1,\pi_1(B))))
 $$
 for all $i_1,...,i_n\in\{0,1,...,N_1-1\}$. This shows that $\nu = \mu_1$ and therefore $\mu_1(E) = \mu_1^{\dagger}(\Psi(E))$ for any Borel set $E$. Consider the set
 $$
 {\mathcal N} = \{x\in T(R_1,\pi_1(B)): \Gamma_2 \ \mbox{is not a spectrum} \ \mbox{for} \ \mu_x^2 \}
 $$
 Then
 $$
 \Psi({\mathcal N}) = \{\Psi(x)\in X_1^{\dagger}: \Gamma_2 \ \mbox{is not a spectrum} \ \mbox{for} \ \mu_x^2 \}
 $$
Note also that, on the second component, the two pairs $(R,B)$ and $(R^\dagger,B^\dagger)$ are the same, more precisely $X_2(x)=X_2^\dagger(\Psi(x))$ and $\mu_x^2={\mu^\dagger}_{\Psi(x)}^2$ for all $x\in X_1$. This means that
$$
 \Psi({\mathcal N}) = \{\Psi(x)\in X_1^{\dagger}: \Gamma_2 \ \mbox{is not a spectrum} \ \mbox{for} \ \mu_{\Psi(x)}^2 \}
$$
which has $\mu_1^{\dagger}$-measure 0, by the arguments in the previous paragraph. Hence, $\mu_1(E) = \mu_1^{\dagger}(\Psi(E))=0$ and this completes the proof.
\end{proof}

\medskip

\begin{proof}[Proof of Theorem \ref{thmain}]

To prove Theorem \ref{thmain},  we use induction on the dimension $d$. We know from \cite{LaWa02,DJ06} that the result is true in dimension one (See also Theorem \ref{example5.0}). Assume it is true for any dimensions less than $d$.

\medskip

First, after some conjugation as in Proposition \ref{przrb}, we can assume that $\bz[R,B]=\bz^d$. Next, if the set $\mathcal Z=\emptyset$, then the result follows from Theorem \ref{th1.1}. Suppose now that ${\mathcal Z}\neq\emptyset$. Then, by Proposition \ref{pr1.14} and Theorem \ref{th_quasi}, we can conjugate with some matrix so that $(R,B)$ are of the quasi-product form given in \eqref{R_4.1} and \eqref{eq1.23.1}.

\medskip

By Theorem \ref{th_quasi}, $(R_1,\pi_1(B), L_1(\ell_2))$ forms a Hadamard triple with some $L$ on ${\mathbb R}^{r}$ where $1\leq r<d$. By induction hypothesis, the measure $\mu_1$ is spectral.
Let $\Lambda_1$ be a spectrum for $\mu_1$. By Proposition \ref{lem1.25}, there exists $\Gamma_2$ such that $\Gamma_2$ is a spectrum for $\mu_x^2$ for $\mu_1$-almost everywhere $x$.  Then we have, with \eqref{eq1.25.2}, and Lemma \ref{lem1.24},
$$\sum_{\gamma_2\in\Gamma_2}\sum_{\lambda_1\in\Lambda_1}|\widehat\mu(x+\lambda_1,y+\gamma_2)|^2=\int_{X_1}\sum_{\gamma_2\in\Gamma_2}|\widehat{{\mu}_s^2}(y+\gamma_2)|^2\,d\mu_1(s)=\int_{X_1}1d\mu_1(s)=1.$$
This means that $\Lambda_1\times\Gamma_2$ is a spectrum for $\mu$ by Lemma \ref{lemJP} and this completes the whole proof of Theorem \ref{thmain}.
\end{proof}

\medskip

\section{Fourier frames}\label{Section9}

In this section, we discuss how Fourier frames can be constructed if we have the almost-Parseval-frame property (Definition \ref{ALmost_DEF}) for the affine pair $(R,B)$. Indeed, this follows from a similar procedure as in the proof of Theorem \ref{th1.1}. Assuming that the almost-Parseval-frame condition is satisfied, we consider sequences $\epsilon_k$ such that $\sum_k\epsilon_k<\infty$ and let $n_k$ and $J_{n_k}$ be the associated sets satisfying
\begin{equation}\label{J}
(1-\epsilon_k)\sum_{b\in B_{n_k}}|w_b|^2\leq \sum_{\lambda\in J_{n_k}}\left|\sum_{b\in B_{n_k}}\frac{1}{\sqrt{N^{n_k}}}w_be^{-2\pi i \langle R^{-n_k}b, \lambda\rangle} \right|^2\leq (1+\epsilon_k)\sum_{b\in B_{n_k}}|w_b|^2.
\end{equation}
Letting $m_k = n_1+n_2+...+n_{k}$, we consider
$$
\Lambda_k = J_{n_1}+ (R^T)^{m_1}J_{n_2}+ (R^T)^{m_2} J_{n_3}+...+ (R^T)^{m_{k-1}} J_{n_{k}},
\ \Lambda = \bigcup_{k=1}^{\infty}\Lambda_k..
$$
 We have the following theorem analogous to Theorem \ref{prop_main1}.

\medskip

\begin{theorem}
Suppose that $B$ is a simple digit set for $R$. Let $\mu=\mu(R,B)$ be the associated self-affine measure with equal weights. Assume that the almost-Parseval-frame condition is satisfied and that
$$
\delta (\Lambda) := \inf_{k}\inf_{\lambda\in\Lambda_k}|\widehat{\mu}((R^T)^{-m_k}\lambda)|^2>0
$$
Then the set $E(\Lambda): = \{e^{2\pi i \langle\lambda,x\rangle}:\lambda\in \Lambda\}$ is a Fourier frame for $L^2(\mu)$ with
\begin{equation}\label{Ffrane}
c\delta(\Lambda)\|f\|^2\leq \sum_{\lambda\in\Lambda}\left|\int f(x)e^{-2\pi i \lambda x}d\mu(x)\right|^2\leq C\|f\|^2
\end{equation}
where $c = \prod_{j=1}^{\infty}(1-\epsilon_j)$ and $C = \prod_{j=1}^{\infty}(1+\epsilon_j)$.
\end{theorem}

\medskip

\begin{proof}
The proof is analogous to the proof of Theorem \ref{prop_main1}. We check that the step functions on the self-affine sets ${\mathcal S}: = \bigcup_{k=1}^{\infty}{\mathcal S}_k$   satisfies the frame inequality. By Proposition \ref{Prop_Approach_2.1}, for any $k\geq 1$
$$
c_k\sum_{b\in B_{m_{k}}}|w_b|^2\leq \sum_{\lambda\in\Lambda_k}\left| \frac{1}{\sqrt{N^{m_k}}}\sum_{b\in B_{m_k}}w_be^{-2\pi i \langle R^{-m_k}b,\lambda\rangle}\right|^2\leq C_k\sum_{b\in B_{m_{k}}}|w_b|^2
$$
where $c_k = \prod_{j=1}^{k}(1-\epsilon_j)$ and $C_k = \prod_{j=1}^{k}(1+\epsilon_j)$. In view of Lemma \ref{lem1},
$$
\sum_{\lambda\in\Lambda_{m_{k}}}\left|\int f(x)e^{-2\pi i \langle\lambda, x\rangle}d\mu(x)\right|^2 =\frac{1}{N^{m_k}}\left| \frac{1}{\sqrt{N^{m_k}}}\widehat{\mu}((R^{T})^{-{m_k}}\lambda)\sum_{b\in B_{m_{k}}}w_de^{-2\pi i \langle R^{-n}b,\lambda\rangle}\right|^2.
$$
As $\delta(\Lambda)\leq |\widehat{\mu}((R^T)^{-{m_k}}\lambda)|^2\leq 1$, Lemma \ref{lem1} implies that this term is bounded above by $C\|f\|^2$ and bounded below by $c\delta(\Lambda)\|f\|^2$,
$$
c\delta(\Lambda)\|f\|^2\leq\sum_{\lambda\in\Lambda_{m_{k}}}\left|\int f(x)e^{-2\pi i \langle\lambda, x\rangle}d\mu(x)\right|^2\leq C\|f\|^2.
$$
But since ${\mathcal S}_{m_k}\subset{\mathcal S}_{m_{\ell}}$ for any $\ell\geq k$, we will have
$$
c\delta(\Lambda)\|f\|^2\leq\sum_{\lambda\in\Lambda_{m_{\ell}}}\left|\int f(x)e^{-2\pi i \langle\lambda, x\rangle}d\mu(x)\right|^2\leq C\|f\|^2, \ f\in{\mathcal S}_{m_k}.
$$
This shows the frame inequality holds by letting $\ell$ go to infinity.
\end{proof}

\medskip

\begin{proof}[Proof of Theorem \ref{th1.2}.]
It suffices to show that under the assumption that ${\mathcal Z}=\emptyset$, we can find some $\Lambda$ such that $\delta(\Lambda)>0$. This proof will be analogous to Proposition \ref{prop_main2}.

\medskip

 Let $\overline{L}$ be a complete set of representatives (mod $R^T({\mathbb Z}^d)$) and let $X= T(R^T,\overline{L})$, the self-affine tile generated by $R^T$ and $\overline{L}$. Since the almost-Parseval-frame condition is satisfied, we can pick the sets $J_{n_i}$ as in \eqref{J}, with  bounds $1-\epsilon_i,1+\epsilon_i$ and $\sum_{i}\epsilon_i<\infty$. The elements of ${J_{n_i}}$ are in distinct residue classes (mod $(R^T)^{n_i}{\mathbb Z}^d$) by Proposition \ref{proposition2.1}(i). By Proposition \ref{proposition2.1}(ii), we may assume $J_{n_i}\subset \overline{L}+R^T \overline{L}+...+(R^T)^{n_i-1}\overline{L}$.  Thus, by the definition of $X$,
$$
(R^{T})^{-(n_i+p)}J_{n_i}\subset X ,\quad(p\geq 0).
$$
Using this $X$, the rest of the proof is the same as in Proposition \ref{prop_main2}.
\end{proof}

\medskip

Next we present some sufficient geometric conditions that guarantee that $\mathcal Z=\emptyset$.

\begin{definition}\label{T-SOSC}
We say that the IFS $\{\tau_b\}_{b\in B}$ satisfies the {\it $T$-strong open set condition} (denoted in short by {\it $T$-SOSC}) if there exists a complete set of representatives $(\mod R(\bz^d))$, $\overline{B}$ such that $B\subset\overline{B}$, $T(R,B)\cap T^\circ\neq \emptyset$ where $T = T(R,\overline{B})$. ($T^\circ$ denotes the interior).
\end{definition}

As shown in \cite{LW1}, if $\overline B$ is a complete set of representatives $(\mod R(\bz^d))$ then $T(R,\overline B)$ tiles $\br^d$ by some lattice so it is a {\it self-affine} tile. In particular, we say that $T$ is a {\it ${\mathbb Z}^d$-tile} if $T$ is a translational tile with ${\mathbb Z}^d$ as a tiling set.

\medskip

\begin{theorem}\label{th9.1}
	Let $\mu=\mu(R,B)$ be the associated self-affine measure. Consider the following conditions:	
	\begin{enumerate}
   \item The affine IFS associated with $R$ and $B$  satisfies the ($T$-SOSC) with $T$ a ${\mathbb Z}^d$-tile.
		\item For all $k\neq k'$ in $\bz^d$, $\mu((T(R,B)+k)\cap(T(R,B)+k'))=0$.
		\item The set ${\mathcal Z} =\emptyset.$
	\end{enumerate}
Then we have the following implications (i) $\Rightarrow$ (ii) $\Rightarrow$ (iii).
   \end{theorem}

\medskip

\begin{proof}

Suppose that (i) holds. Suppose that ($T$-SOSC) is satisfied for $\{\tau_b\}_{b\in B}$ with $T=T(R,\overline B)$. By the invariance property of $T$, we get that the SOSC condition is satisfied with open set $T^\circ$. By Theorem \ref{proposition2.2} taking  $U=T^\circ$,  we get that $\mu(T^\circ)=1$, $\mu(\partial T)=0$ and that the IFS satisfies the no overlap condition holds. We note that for all $n\in \bz^d$, $\mu(\partial T + n) = \mu((\partial T +n)\cap T^\circ) = 0$ (as $T$ is a ${\mathbb Z}^d$-tile so $(\partial T+n)\cap T^\circ=\ty$, see \cite{LW1}). Moreover, $T(R,B)\subset T$ and $T$ tiles by $\bz^d$ implies that for any $n\neq n'$ in $\bz^d$,
$$
(T(R,B)+n)\cap (T(R,B)+n')\subseteq (T+n)\cap (T+n') = ( \partial T+n)\cap (\partial T+n').
$$
Hence,  $\mu((T(R,B)+n)\cap (T(R,B)+n')) \leq \mu (( \partial T+n)\cap (\partial T+n'))=0$.

\medskip

If (ii) holds, consider the set
$$\mathcal N:=\{x\in T(R,B): \mbox{ There exists $y\in T(R,B)$, $y\neq x$ such that $e^{2\pi i\langle n,x-y\rangle}=1$ for all $n\in{\mathbb Z}^d$}\}$$
$$=\{x\in T(R,B) : \mbox{ There exists $y\in T(R,B)$ such that $-\alpha:=x-y\in{\mathbb Z}^d$}\}$$
$$=\bigcup_{\alpha\in{\mathbb Z}^d}\{x\in T(R,B): x+\alpha\in T(R,B)\} = \bigcup_{\alpha\in{\mathbb Z}^d}(T(R,B)\cap (T(R,B)-\alpha)).$$
By hypothesis,  $\mathcal N$ has measure zero.

\medskip
Now take $\mathcal K$ to be an arbitrary compact subset of $T(R,B)\setminus\mathcal N$. The collection of exponential functions $ E({\mathbb Z}^d) : =\{e^{2\pi i \langle n,x\rangle} : n\in{\mathbb Z}^d\}$ separates points in $\mathcal K$, therefore, by Stone-Weierstrass theorem, we get that $E({\mathbb Z}^d)$ spans $L^2(\mathcal K,\mu)$, and since $\mathcal K$ was arbitrary close to $T(R,B)$ in measure, we get that these exponentials span $L^2(T(R,B),\mu)$. Hence, for $\xi\in {\mathbb R}^d$ we cannot have $\widehat\mu(\xi+ n)=0$ for all $n\in{\mathbb Z}^d$, because that would imply that $e^{2\pi i \langle \xi,x\rangle}$ is orthogonal to all $e^{2\pi i \langle n,x\rangle}$ for all $ n\in{\mathbb Z}^d$, which contradicts the completeness. This shows ${\mathcal Z}=\emptyset$.
\end{proof}

\medskip

The previous theorem leads us to a simple corollary in ${\mathbb R}^1$.
\begin{corollary}
Let $N\ge 2$ be a positive integer. Suppose that $B\subset \left\{0,1,...,N-1\right\}$. Consider the IFS $\tau_b(x) = \frac1N(x+b), \ b\in B$. Suppose that the almost-Parseval-frame condition is satisfied for this self-similar IFS. Then the corresponding (equal-weighted) self-similar measure admits a Fourier frame.
\end{corollary}

\begin{proof}
For the IFS, the attractor is contained in $[0,1]$, which is the self-similar set generated by $\overline{B} = \{0,1,...,N-1\}$. Hence, (T-SOSC) is satisfied. The conclusion follows from Theorem \ref{th9.1} and Theorem \ref{th1.2}.
\end{proof}

In particular, the Middle-Third Cantor measure satisfies (T-SOSC), hence ${\mathcal Z}=\emptyset$ is satisfied.

\medskip

\section{Open problems}
One of the major open problems in the study of Fourier analysis on fractals is to see whether the non-spectral self-affine measures are still frame-spectral (See (Q2) in the introduction). The idea of almost-Parseval-frame towers turns this problem into a problem of matrix analysis. Given an integral expanding matrix $R$ and a set of simple digits $B$ with $N=\#B<|\det R|$, the almost-Parseval-frame conditions can be reformulated equivalently as {\it for any $\epsilon>0$, there exists $n\in{\mathbb N}$ and a set of $L_n\subset {\mathbb Z}^d$ such that the matrix}
$$
F_n (B_n,L_n) = \frac{1}{\sqrt{N^n}}\left(e^{2\pi i \langle R^{-n} b,\ell\rangle}\right)_{\ell\in L_n,b\in B_n}
$$
{\it satisfies}
$$
(1-\epsilon)\|{\bf w}\|^2\le \|F_n{\bf w}\|^2\le(1+\epsilon)\|{\bf w}\|^2
$$
{\it for any vectors ${\bf w}\in {\mathbb C}^{N^n}$. (Recall that $B_n = B+RB+...+R^{n-1}B$)}

\medskip

We observe that if we let ${\overline{B_n}}$ and $\overline{L_n}$ be respectively the complete representative class  (mod $R^n({\mathbb Z}^d))$ and  (mod $(R^T)^n({\mathbb Z}^d))$. Then the matrix
$$
\overline{F_n}: = \frac{1}{\sqrt{|\det R|^n}}\left(e^{2\pi i \langle R^{-n} b,\ell\rangle}\right)_{\ell\in \overline{L_n},b\in \overline{B_n}}
$$
forms a unitary matrix. i.e.
$$
\|\overline{F_n}{\bf w}\| = \|{\bf w}\|, \ \forall {\bf w}\in{\mathbb C}^{|\det R|^n}
$$
 As $B_n\subset {\overline{B_n}}$, we can take the vectors ${\bf w}$ such that they are zero on the coordinates which are not in $B_n$. This implies that
 $$
\| F_n(B_n,\overline{L_n}){\bf w}\| = \frac{|\det R|^n}{N^n}\|{\bf w}\|.
 $$
 In other words,
 $$
 \sum_{\lambda\in \overline{L_n}}\left|\sum_{b\in B_n} w_b \frac{1}{\sqrt{N^n}}e^{-2\pi i \langle R^{-n}b,\lambda\rangle}\right|^2 = \frac{|\det R|^n}{N^n}\sum_{b\in B_n}|w_b|^2.
 $$
 This shows that the collection of vectors $\{ \left(\frac{1}{\sqrt{N^n}}e^{-2\pi i \langle R^{-n}b,\lambda\rangle}\right)_{b\in  B_n}:\lambda\in \overline{L_n}\}$ forms a tight frame for ${\mathbb C}^{N^n}$ with frame bound $\frac{|\det R|^n}{N^n}$. Our problem is to extract a subset $L_n$ from $\overline{L_n}$ such that we have an almost tight frame with frame constant nearly 1. This reminds us of the Kadison-Singer problem that was open for over 50 years and solved recently in \cite{MSS}.

 \medskip

\begin{theorem}\cite[Corollary 1.5]{MSS}
Let $r$ be a positive integer and let $u_1,...,u_m\in{\mathbb C}^d$ such that
$$
\sum_{i=1}^m|\langle w, u_i\rangle|^2 = \|w\|^2 \ \forall w\in {\mathbb C}^d
$$
and $\|u_i\|\leq \delta$ for all $i$. Then there exists a partition $S_1,...,S_r$ of $\{1,...,m\}$ such that
$$
\sum_{i\in S_j}|\langle w, u_i\rangle|^2 \le  \left(\frac{1}{\sqrt{r}}+\sqrt{\delta}\right)^2\|w\|^2 \ \forall w\in {\mathbb C}^d.
$$
 \end{theorem}

 \medskip

 This statement says that we can partition a tight frame into $r$ subsets such that the frame constant of each partition is almost $1/r$. Iterating this process allowed Nitzan et al \cite{NOU} to establish the existence of Fourier frames on any unbounded sets of finite measure. One of their lemmas states:

\begin{lemma}\label{lem6.1}
\cite[Lemma 3]{NOU} Let $A$ be an $K\times L$ matrix and $J\subset \{1,...,K\}$, we denote by $A(J)$ the sub-matrix of $A$ whose rows belong to the index $J$. Then there exist universal constants $c_0,C_0>0$ such that whenever $A$ is a $K\times L$ matrix, which is a sub-matrix of some $K\times K$ orthonormal matrix, such that all of its rows have equal $\ell^2$-norm, one can find a subset $J\subset\{1,...,K\}$ such that
$$
c_0\frac{L}{K}\|{\bf w}\|^2\leq \|A(J){\bf w}\|^2\leq C_0\frac{L}{K}\|{\bf w}\|^2 ,\ \forall {\bf w}\in {\mathbb C}^n.
$$
\end{lemma}

\medskip
This lemma leads naturally to the following:

\begin{proposition}\label{prop1.3a}
With $(R,B)$ as in Definition \ref{defifs}, there exist  universal constants $0< c_0< C_0<\infty$ such that for all $n$, there exists $J_n$ such that
$$
c_0\sum_{b\in B_n}|w_b|^2\leq \sum_{\lambda\in J_n}\left|\frac{1}{\sqrt{N^n}}\sum_{b\in B_n}w_be^{-2\pi i \langle R^{-n}b, \lambda\rangle}\right|^2\leq C_0\sum_{b\in B_n}|w_b|^2
$$
for all $(w_b)_{b\in B_n}\in{\mathbb C}^{N^n}$.
\end{proposition}

\begin{proof} Let
$$
F_n = \frac{1}{|\det R|^{n/2}}\left[e^{2\pi i \langle R^{-n}b,\ell\rangle}\right]_{\ell\in \overline{L}_n,b\in\overline{B}_n}
$$
where $\overline{B}_n$ is a complete coset representative (mod $R({\mathbb Z}^d)$) containing $B_n$ and  $\overline{L}_n$ is a complete coset representative (mod $R^T({\mathbb Z}^d)$). It is well known that ${\mathcal F}_n$ is an orthonormal matrix. Let $K = |\det R|^n$ and
$$
A_n = \frac{1}{|\det R|^{n/2}}\left[e^{2\pi i \langle R^{-n}b,\ell\rangle}\right]_{\ell\in \overline{L}_n,b\in B_n}.
$$
Then $A_n$ is a sub-matrix of $F_n$ whose columns are exactly the ones with index in $B_n$ so that the size $L$ is $L = N^n$. By Lemma \ref{lem6.1}, we can find universal constants $c_0,C_0$, independent of $n$, such that for some $J_n\subset \overline{L}_n$, we have
$$
c_0\frac{N^n}{|\det R|^n}\|{\bf w}\|^2\leq \|A(J_n){\bf w}\|^2\leq C_0\frac{N^n}{|\det R|^n}\|{\bf w}\|^2 ,\ \forall {\bf w}\in {\mathbb C}^{N^n}.
$$
As $\frac{|\det R|^{n/2}}{N^{n/2}}A(J_n) = \frac{1}{|\det R|^{n/2}}\left[e^{2\pi i \langle R^{-n}b,\ell\rangle}\right]_{\ell\in J_n,b\in B_n}: = F_n$, this shows
$$
c_0\|{\bf w}\|^2\leq \|F_n{\bf w}\|^2\leq C_0\|{\bf w}\|^2 ,\ \forall {\bf w}\in {\mathbb C}^{N^n}.
$$
This is equivalent to the inequality we stated.
\end{proof}

Since the bounds $c_0$ and $C_0$ are not close to 1, we cannot use the same procedure and ``concatenate'' the sets $J_n$ as in \eqref{eqLambda_k}, as in the proof of Theorem \ref{th1.2}; if we do so for Proposition \ref{prop1.3}, the upper bound grows to infinity and the lower bound decreases to 0. However, this can be circumvented if we can make construct $J_n$ {\it increasingly}.\medskip

\medskip

Another central problem in the study of spectral measure is  the converse of Theorem \ref{thmain}.

\medskip

{\bf (Q3):} Suppose that $\mu(R,B)$ is a spectral measure, does there exist $L$ such that $(R,B,L)$ forms a Hadamard triple?

\medskip

This question suggests that the only way to produce spectral self-affine measures is the first level discrete measure is spectral. The question is directly related to  the Laba-Wang conjecture \cite{LaWa02}. There are several difficulties one has to overcome in order to make progress in this problem. Maybe the first one is if self-affine spectral measures can have overlap. All known fractal spectral measures have no overlap. Perhaps the following questions have a positive answer.

\medskip

{\bf (Q4):} Suppose that $\mu(R,B)$ is a spectral measure, is it true that there is no overlap?

\medskip

{\bf (Q5):} Suppose that $\mu(R,B)$ is a spectral measure and there is no overlap, is there a set $L$ such that $(R,B,L)$ is a Hadamard triple?

\medskip

\medskip

\section{Appendix}
In this appendix, we follow all the same notation in Section 9. Our aim is to prove that the constant $\delta(\Lambda)$ does not appear in the lower bound, showing also that the frame inequality in \eqref{eq2.1i} in Theorem \ref{prop_main1} is indeed the Parseval identity.

\medskip

\begin{theorem}\label{th3.3}
Suppose that  $\delta(\Lambda)>0$. Then $\mu(R,B)$ admits a Fourier frame $E(\Lambda)$ with  lower and upper frame bounds respectively equal
$$
\prod_{j=1}^{\infty}(1-\epsilon_j), \prod_{j=1}^{\infty}(1+\epsilon_j).
$$
In particular, if all $\epsilon_j =0$, then $\mu$ admits a tight Fourier frame.
\end{theorem}

\begin{proof}
This proof generalizes from the proofs in \cite{MR3055992}. It suffices to show that the Fourier frame inequality holds for a dense set of functions in $L^2(\mu)$, in which we will check it for step functions in ${\mathcal S}$. Let $f = \sum_{{\bf b}\in {\bf B}_n}w_{\bf b}{\bf 1}_{T(R,B)_{\bf b}}\in {\mathcal S}_n$ and
$$
Q (f) = \sum_{\lambda\in\Lambda}\left|\int f(x)e^{-2\pi i \lambda x}d\mu(x)\right|^2  =\lim_{n\rightarrow\infty}Q_n(f),
$$
where
$$
Q_n(f): = \sum_{\lambda\in \Lambda_n}\left|\int f(x)e^{-2\pi i \lambda x}d\mu(x)\right|^2 = \frac{1}{N^n}\sum_{\lambda\in\Lambda_n}|\widehat{\mu}((R^{T})^{-n}\lambda)|^2\left|\sum_{{\bf b}\in B_n}w_{\bf b} \frac{1}{\sqrt{N^n}}e^{-2\pi i  \langle R^{-n}{\bf b}, \lambda\rangle}\right|^2.
$$
 Let $C_n = \prod_{j=1}^{n}(1-\epsilon_j)$ and $D_n = \prod_{j=1}^{n}(1+\epsilon_j)$ for $n=1,2...$ and $n=\infty$. We first consider the upper bound. Considering $m>n$, we have $f\in {\mathcal S}_m$. Thus,
$$
\begin{aligned}
Q_m(f) =& \sum_{\lambda\in \Lambda_m}\left|\int f(x)e^{-2\pi i \lambda x}d\mu(x)\right|^2\\
=&\sum_{\lambda\in \Lambda_m}\frac{1}{N^m}|\widehat{\mu}((R^T)^{-m}\lambda)|^2\left|\sum_{{\bf b}\in B_m}w_{\bf b} \frac{1}{\sqrt{N^m}}e^{-2\pi i  \langle R^{-m}{\bf b}, \lambda\rangle}\right|^2.\\
\leq& \frac{1}{N^m}\sum_{\lambda\in \Lambda_m}\left|\sum_{{\bf b}\in B_m}w_{\bf b} \frac{1}{\sqrt{N^m}}e^{-2\pi i \langle R^{-m}{\bf b}, \lambda\rangle}\right|^2.\\
=&\frac{1}{N^m}D_m\sum_{{\bf b}\in B_m}|w_{\bf b}|^2  \le D_{\infty} \int|f|^2d\mu. \\.
\end{aligned}
$$
This shows the upper bound by taking $m$ to infinity. The lower bound requires some more work.
$$
\begin{aligned}
Q_m(f) =& Q_n(f)+ \sum_{\lambda\in\Lambda_m\setminus \Lambda_{n}}\left|\int f(x)e^{-2\pi i \langle\lambda, x\rangle}d\mu(x)\right|^2\\
=&Q_n(f)+\sum_{\lambda\in\Lambda_m\setminus \Lambda_{n}}\frac{1}{N^m}|\widehat{\mu}((R^T)^{-m}\lambda)|^2\left|\sum_{{\bf b}\in B_m}w_{\bf b} \frac{1}{\sqrt{N^m}}e^{-2\pi i  \langle R^{-n}{\bf b}, \lambda\rangle}\right|^2.\\
\ge& Q_n(f)+\delta(\Lambda)\cdot\sum_{\lambda\in \Lambda_m\setminus \Lambda_{n}}\frac{1}{N^m}\left|\sum_{{\bf b}\in B_m}w_{\bf b} \frac{1}{\sqrt{N^m}}e^{-2\pi i  \langle R^{-n}{\bf b}, \lambda\rangle}\right|^2.\\
\end{aligned}
$$
Note that
$$
\sum_{\lambda\in \Lambda_m}\frac{1}{N^m}\left|\sum_{{\bf b}\in B_m}w_{\bf b} \frac{1}{\sqrt{N^m}}e^{-2\pi i  \langle R^{-n}{\bf b}, \lambda\rangle}\right|^2 \ge C_m\int|f|^2d\mu.
$$
We further have
$$
\begin{aligned}
Q_m(f) \ge& Q_n(f)+\delta(\Lambda)\cdot \left(C_m\int|f|^2d\mu-\sum_{\lambda\in \Lambda_n}\frac{1}{N^m}\left|\sum_{{\bf b}\in B_m}w_{\bf b} \frac{1}{\sqrt{N^m}}e^{-2\pi i  \langle R^{-n}{\bf b}, \lambda\rangle}\right|^2\right)\\
\ge &Q_n(f)+\delta(\Lambda)\cdot \left(C_m\int|f|^2d\mu-\sum_{\lambda\in \Lambda_n}\left|\int f(x)e^{-2\pi i \lambda x}d\mu_m(x)\right|^2\right)
\end{aligned}
$$
For a fixed $n$, we take $m$ to infinity. By the fact that $Q_{m}(f)$ converges to $Q_{\infty}(f)$ and $\mu_m$ converges weakly to $\mu$, we have
$$
Q_{\infty}(f)\ge Q_n(f) +\delta(\Lambda)\cdot \left(C_{\infty}\int|f|^2d\mu- \sum_{\lambda\in \Lambda_n}\left|\int f(x)e^{-2 \pi i \lambda x}d\mu(x)\right|^2\right).
$$
We then take $n$ to infinity eventually and obtain
$$
Q_{\infty}(f)\ge Q_{\infty}(f) +\delta(\Lambda)\cdot \left(C_{\infty}\int|f|^2d\mu-\sum_{\lambda\in \Lambda}\left|\int f(x)e^{-2\pi i \lambda x}d\mu(x)\right|^2\right).
$$
and thus
$$
\delta(\Lambda)\cdot \left(C_{\infty}\int|f|^2d\mu-\sum_{\lambda\in \Lambda}\left|\int f(x)e^{-2\pi i \lambda x}d\mu(x)\right|^2\right)\leq 0.
$$
However, $\delta(\Lambda)>0$ and we have
$$
C_{\infty}\int|f|^2d\mu\le \sum_{\lambda\in \Lambda}\left|\int f(x)e^{-2\pi i \lambda x}d\mu(x)\right|^2
$$
This establishes the lower bound. If all $\epsilon_j = 0$, we will have $C_m = D_m$ for all $m$. Taking $m$ to infinity, $C_{\infty} = D_{\infty}$ and hence we have a tight Fourier frame.
\end{proof}

 \begin{acknowledgements}
This work was partially supported by a grant from the Simons Foundation (\#228539 to Dorin Dutkay) and Chun-Kit Lai was supported by the mini-grant by ORSP of San Francisco State University (Grant No: ST659).
\end{acknowledgements}

\bibliographystyle{abbrv}
\bibliography{eframes}

\end{document}